\documentclass[a4paper]{amsart}
\usepackage{amsmath}
\usepackage{amsthm}
\usepackage{amsfonts}

\usepackage[alphabetic]{amsrefs}
\usepackage[T1]{fontenc}
\usepackage[utf8]{inputenc}
\usepackage{lmodern}
\usepackage[english]{babel}
\usepackage{csquotes}

\usepackage{tikz}
\usepackage{tikz-cd}
\usetikzlibrary{knots}

\usepackage{adjustbox}
\usepackage{inconsolata}

\usepackage{mathtools}
\usepackage{oubraces}
\usepackage{amsthm}
\usepackage{amsbsy}
\usepackage{amssymb}
\usepackage{amsthm}
\usepackage{caption}

\theoremstyle{plain}

\newtheorem{conj}{Conjecture}[section]

\newtheorem{thm}[conj]{Theorem}
\newtheorem{lemma}[conj]{Lemma}

\newtheorem{definition}[conj]{Definition}
\newtheorem{corollary}[conj]{Corollary}
\theoremstyle{remark}
\newtheorem{remark}{Remark}[section]
\numberwithin{equation}{section}
\numberwithin{figure}{section}

\title{Formal groups over non-commutative rings}
\author{Christian Nassau}
\email{nassau@nullhomotopie.de}
\date{\today}


\DeclareMathOperator*{\id}{id}

\newcommand{\ZZ}{{\mathbb Z}}

\newcommand{\CC}{\mathbb{C}}
\newcommand{\CP}{\mathbb{C}P}
\newcommand{\BFK}{\mathcal{B}}

\newcommand{\LC}{{\mathcal L}}

\newcommand{\Tria}{\mathcal{Y}}

\newcommand{\TriSpace}{\mathrm{Tri}}
\newcommand{\opposite}[1]{{#1}^{\mathrm{op}}}
\DeclareMathOperator{\Map}{Map}

\DeclareMathOperator{\Hom}{Hom}

\DeclareMathOperator{\Mod}{Mod}
\DeclareMathOperator{\Sym}{Sym}
\DeclareMathOperator{\QSym}{QSym}

\DeclareMathOperator{\NSym}{NSym}
\DeclareMathOperator{\Ass}{Ass}
\DeclareMathOperator{\UAss}{\mathcal{U}_{ass}}

\DeclareMathOperator*{\cotensor}{\Box}
\DeclareMathOperator*{\Tor}{Tor}

\DeclareMathOperator*{\len}{len}

\newcommand{\hur}{\mathrm{hur}}

\begin{document}

\begin{abstract}
We develop an extension of the usual theory of formal group laws
where the base ring is not required to be commutative
and where the formal variables need neither be central nor have to commute with each other.

We show that this is the natural kind of formal group law for the needs of
algebraic topology in the sense that a (possibly non-commutative)
complex oriented ring spectrum is canonically equipped with just such a formal group law.
The universal formal group law is carried by the Baker-Richter spectrum $M\xi$
which plays a role analogous to $MU$ in this non-commutative context.

As suggested by previous work of Morava the
Hopf algebra $\BFK$
of \enquote{formal diffeomorphisms of the non-commutative line}
of Brouder, Frabetti and Krattenthaler
is central to the theory developed here.
In particular, we verify Morava's conjecture that
there is a representation of the Drinfeld quantum-double $D(\BFK)$
through cohomology operations in $M\xi$.
\end{abstract}

\maketitle
\tableofcontents


\begin{section}{Introduction}
This paper is about non-commutative complex oriented cohomology theories:
these are functors $X\mapsto E^\ast X$ from spaces to graded rings that satisfy the Eilenberg-Steenrod axioms 
for cohomology, except possibly for the normalization axiom that governs the cohomology of a one-point space 
$E^\ast = E^\ast(\{\mathrm{pt.}\})$.
Such a cohomology theory is called \emph{complex oriented} if every 
complex line bundle $\LC \rightarrow X$ has an associated first Chern class $c_1^E(\LC)\in E^2(X)$.
The Chern class of the tensor product of two line bundles $\LC_1$, $\LC_2$ can then be expressed 
as \[ c_1^E(\LC_1\otimes \LC_2) = c_1^E(\LC_1) +_F c_1^E(\LC_2) 
= c_1^E(\LC_1) + c_1^E(\LC_2) + \sum_{i,j\ge 1} a_{i,j} c_1^E(\LC_1)^i c_1^E(\LC_2)^j. \]
Here $x+_Fy = x + y + \sum a_{i,j} x^i y^j$ is a formal addition law, 
by its geometric provenance
necessarily associative, commutative and unital, with coefficients $a_{i,j}\in E^\ast$.
This is the \emph{formal group law} associated to $E$
and its importance can hardly be overstated.

The classical theory considers only \emph{commutative} complex oriented cohomology theories.
For those the formal group law can conveniently be discussed in the language of formal algebraic geometry
for which there exists an extensive body of work (see \cite{MR2987372}).
Indeed, we believe that no example of a significantly\footnote{%
Morava $K$-theory at the prime $2$ is an example of a non-commutative complex oriented ring spectrum that has been studied a lot. 
But the non-commutativity can only be seen by working with odd-dimensional cohomology classes. A significantly non-commutative ring spectrum would be one where the coefficient ring itself and/or its ring of Chern classes is non-commutative.} 
non-commutative complex oriented cohomology theory
exists in the literature - with the single exception of the Baker-Richter spectrum $M\xi$ of \cite{MR2484353} and \cite{MR3243390}.

We here initiate a study of the formal group laws that can arise from non-commutative cohomology theories.
These deviate from the classical commutative ones in several points:
\begin{enumerate}
    \item The ring of definition $E_\ast$ can be non-commutative.
    \item Adjoining a formal variable $x$ by looking at $E^\ast \CP^\infty = E^\ast[[x]]$ gives an $x$ that is not central: there will be a non-trivial commutation relation between this $x$ and the scalars from $E_\ast$.
    \item Adjoining a second formal variable $y$ by looking at $E^\ast (\CP^\infty \times \CP^\infty)= E^\ast[[x,y]]$ gives a $y$ that does not commute with either the scalars from $E_\ast$ nor with the first variable $x$: there will be a non-trivial commutation relation between $x$ and $y$.
    \item Complex vector bundles do admit a theory of Chern classes since 
    $E^\ast BU(n)$ can still be expressed as $E^\ast[[c_1,c_2,\ldots,c_n]]$ by the collapsing of the Atiyah-Hirzebruch spectral sequence. But these Chern-classes will not be central, nor will they commute among each other. Their definition is also emburdened with choices that impact their canonicity; as a consequence there will be no easy Whitney-sum formula for the $c_n(E\oplus F)$. 
\end{enumerate}
We will begin the study of such formal group laws in section~\ref{sec:bfk}
by showing that the commutation relations between scalars and formal variables can be elegantly expressed 
as an action of the Hopf algebra of \enquote{formal diffeomorphisms of the non-commutative line}
defined by Brouder, Frabetti, and Krattenthaler in \cite{MR2200854}.
In that framework the formal group law appears as a braided Hopf algebra structure defined on a ring of formal power series.

We then set course for proving that the formal group law of $M\xi$ is in fact the \emph{universal}
one-dimensional commutative formal group law, just like the formal group law of the 
complex bordism theory $MU$ is the universal  one-dimensional commutative formal group law
\emph{that is defined over a commutative ring}. This is complemented by showing that $M\xi$ is itself universal 
as a complex oriented cohomology theory 
(in the same sense that $MU$ is the universal \emph{commutative} complex oriented cohomology theory).

We spend some time investigating the fate of Chern classes (equivalently, of symmetric functions) 
for a non-commutative $E$. We show that the work of the Gelfand school \cite{MR2132761} on quasi-determinants and non-commutative symmetric functions is highly relevant here: it provides the tools necessary to define the Chern classes.
These classes are essential for understanding the connection between $E$ and the complex cobordism spectrum $MU$.
They define a (non-canonical and non-multiplicative) map $MU\rightarrow E$ 
that becomes useful when we later investigate the connection between $MU$ and $M\xi$:
we show in section~\ref{sec:splitting} that $M\xi$ splits (non-multiplicatively) as a suspension of copies of $MU$. 
More specifically, we show that there is a large free associative subalgebra $\Tria_\ast \subset M\xi_\ast$ 
and a non-multiplicative decomposition $M\xi_\ast \cong \Tria_\ast \otimes MU_\ast$ that lifts to a splitting 
of the spectrum $M\xi$.
(That $M\xi$ splits $p$-locally as a wedge of copies of the Brown-Peterson spectrum $BP$ had already been established in \cite{MR2484353}).

Throughout the work is complemented by as many explicit computations as we could reasonably provide: 
we give formulas for the universal formal group law, for the universal commutation relations 
and for the Chern classes, for example.  
We hope these computations help to better appreciate the nature of this new class of 
formal group law that lives in a non-commutative world.   
The computations were made with the help of the SageMath computer algebra system 
\cite{sagemath}.  The code for these computations will be published separately.

There remains just one question to answer before the main exposition can commence: if there is just one example 
of a significantly non-commutative complex oriented ring spectrum in the literature, then why bother?
Our answer is that we believe that non-commutative complex oriented ring spectra 
have to exist in abundance; they just haven't been discovered yet. 
The one example that is known, $M\xi$, is universal and its formal group law is 
richer and more complicated than the formal group law of $MU$ by an order of magnitude. 
It's inconceivable that there would not be interesting non-commutative complex oriented 
cohomology theories at other heights of the chromatic hierarchy.

All the rings that we use in the sequel are implicitly to be considered as \emph{graded} rings
and where we deal with formal variables we will always assume an implicit completion.
We decided to keep mostly silent about the grading and the completion in the body of the text 
since we felt that being explicit about it would not make the exposition more readable. 
\end{section}


\begin{section}{Formal group laws as braided Hopf algebras}\label{sec:bfk}
\begin{subsection}{The Brouder--Frabetti--Krattenthaler Hopf algebra}
Let $S=\ZZ[[x]]$ be the ring of formal power series in one variable
and let $\BFK = \ZZ\langle \phi_1,\phi_2,\ldots\rangle$ be the free associative
algebra on generators $\phi_k$, $k\ge 1$.
Consider the coaction map
$$\Delta\colon{}S \rightarrow \BFK \otimes S,
\qquad \qquad x \mapsto \sum_{k\ge 0} \phi_k \otimes x^{1+k}
\qquad \text{(with $\phi_0=1$)}.$$
\begin{lemma}\label{lem:first}
    There is a unique Hopf algebra structure on $\BFK$ for which $S$ becomes a
    $\BFK$-comodule algebra (see~\ocite{MR1243637}, 4.1.2) with this coaction.
\end{lemma}
\begin{proof}
It is straightforward to compute the coproducts $\Delta(\phi_n)$:
one has
\begin{align*}
    \Delta (x^n) &= \sum_{i_1,\ldots,i_n\ge 0} \phi_{i_1}\cdots\phi_{i_n}\otimes x^{n+i_1+\cdots+i_n}
    = \sum_{j\ge 0} Q_j^{(n)}(\phi) \otimes x^{n+j},
\intertext{where $Q_j^{(n)}(\phi) = \sum_{i_1+\cdots+i_n=j} \phi_{i_1}\cdots\phi_{i_n}$. Therefore}
    (\id\otimes\Delta) \Delta x &= \sum \phi_n \otimes Q_j^{(n)}(\phi) \otimes x^{n+j}
\intertext{From $(\Delta\otimes\id)\Delta x = (\id\otimes\Delta)\Delta x$ one then finds}
\Delta(\phi_n) &= \sum_{p+q=n} \phi_p \otimes Q_q^{(p)}(\phi).
\end{align*}
That this does indeed define a Hopf algebra structure has been proved in~\cite{MR2200854}.
\end{proof}
The Hopf algebra $\BFK$ was first described by Brouder, Frabetti and Krattenthaler in~\cite{MR2200854}.
It is known as the Hopf algebra of \enquote{formal diffeomorphisms of the non-commutative line}.

Now consider a ring $R$ with a left action of $\BFK$.
We assume that $R$ is a $\BFK$-module algebra,
i.e.~that $\phi(ab) = \sum \phi'(a)\phi''(b)$
for $a,b\in R$, $\phi\in \BFK$ and $\Delta \phi = \sum \phi' \otimes \phi''$.
Under this condition the \enquote{smash product} $R \# S$ (\ocite{MR1243637}, 4.1.3) is defined:
this is the tensor product $R \otimes_\ZZ S$
with 
the multiplication
\[(a\#h) \cdot (b\#k) = \sum a\cdot(h'b) \# h''k.\]

Since $S=\ZZ[[x]]$ one can more explicitly see that $R\# S = R[[x]]$
where $R[[x]]$ is defined to be the free left $R$-module with basis $x^k$, $k\ge 0$.
It here carries an additional $R$-algebra structure
determined by the $\BFK$-action on $R$ via
\[x \cdot r = \sum_{k\ge 0} \phi_k(r) \cdot x^{1+k}.\]

\end{subsection}
\begin{subsection}{Braided Hopf algebras}
We will see below that a formal group law over $R$ can best be described as a certain
braided Hopf algebra structure on $R\# S$. For this we need to recall the definition
of a braided Hopf algebra.
The mandatory reference is
Takeuchi's survey \cite{MR1800719}*{Ch. 5}.

Consider a ground ring $R$ (allowed to be non-commutative) and an $R$-bimodule $T$.
In the spirit
of \cite{MR1800719}*{Def.~5.1}
a braided $R$-Hopf algebra structure on $T$
consists of six $R$-bilinear maps
\begin{align*}
    m &\colon{} T\otimes_R T \rightarrow T,
    & u &\colon{} R\rightarrow T,
    & \Upsilon &\colon{} T\otimes_R T \rightarrow T\otimes_R T,
    \\
    \Delta &\colon{} T \rightarrow T\otimes_R T,
    & \epsilon &\colon{} T\rightarrow R,
    & \chi &\colon{} T\rightarrow T.
\end{align*}
The required conditions are
\begin{enumerate}
    \item $(T,m,u)$ is an $R$-algebra.
    \item $(T,\Delta,\epsilon)$ is an $R$-coalgebra.
    \item $\Upsilon$ satisfies the Yang-Baxter equation.
    \item\label{commcond} $m$, $u$, $\Delta$ and $\epsilon$ \enquote{commute with $\Upsilon$}.
    \item $\epsilon\colon{} T\rightarrow R$ is multiplicative, $u : R\rightarrow T$ is comultiplicative.
    \item $\Delta$ is multiplicative and $m$ is comultiplicative
    with respect to the algebra/coalgebra structure on $T\otimes_RT$ defined by $\Upsilon$,
    i.e.
    $$\Delta \circ m =
    \overunderbraces{&&\br{2}{\text{comultiplication on $T\otimes_RT$}}}%
    {&(m\otimes m) \,&\,(\id\otimes\Upsilon\otimes\id) \,&\,(\Delta\otimes\Delta)}%
    {&\br{2}{\text{multiplication on $T\otimes_RT$}}&}.$$
    \item $\chi$ satisfies the usual antipode relations.
\end{enumerate}
Here the Yang-Baxter equation in (3) is
\begin{equation}
    \label{yangbaxter}
    (\Upsilon\otimes\id)
    (\id\otimes\Upsilon)
    (\Upsilon\otimes\id)
    =
    (\id\otimes\Upsilon)
    (\Upsilon\otimes\id)
    (\id\otimes\Upsilon).
\end{equation}
Condition \ref{commcond} needs some elaboration.
As explained in \cite{MR1800719}*{page 311} commutation of $m$ with $\Upsilon$
translates to the two relations
\begin{equation}\label{ycomm1}
\begin{aligned}
    (m{}\otimes\id)(\id\otimes\Upsilon)(\Upsilon\otimes\id) &= \Upsilon(\id\otimes m) \\
    (\id\otimes{}m)(\Upsilon\otimes\id)(\id\otimes\Upsilon) &= \Upsilon(m\otimes\id).
\end{aligned}
\end{equation}
According to \cite{MR1273649}*{Prop.~2.2} these conditions ensure that the multiplication
$m_2 = (m\otimes m)(\id\otimes\Upsilon\otimes\id)$ on $T\otimes_RT$ is associative.
Likewise, commutation of $\Upsilon$ and $u$ means
\begin{equation}\label{ycomm2}
    \Upsilon(u\otimes\id) = \id\otimes u, \quad \Upsilon(\id\otimes u) = u \otimes\id
\end{equation}
and these imply that $m_2$ is unital with unit $u\otimes u$ (\cite{MR1273649}*{Prop.~2.3}).

We can now say in simple (though probably somewhat opaque) words what
a formal group law over $R$ is:
\begin{definition}\label{def:fg}
Let $R$ be a $\BFK$-module algebra. Write $T=R\# S$ and let $m$ and $u$ denote the
multiplication resp.~unit map of $T$. Let $\epsilon \colon{} T\rightarrow R$
be given by
\[\epsilon(r\# 1) = r,\qquad \epsilon(r\#x^{1+k}) = 0.\]
A \emph{one-dimensional commutative formal group law} over $R$ consists of
three $R$-bilinear maps
\[\Delta \colon{} T \rightarrow T \otimes_R T, \qquad
\Upsilon \colon{} T \otimes_R T \rightarrow T \otimes_R T, \qquad
\chi \colon{} T\rightarrow T\]
such that $(T,m,u,\Upsilon,\Delta,\epsilon,\chi)$
is a commutative and cocommutative braided $R$-Hopf algebra.
\end{definition}
We have several remarks that help to
better understand why this structure
really represents a formal group law.

\begin{remark}
    The $\BFK$-action on $R$ turns $R[[x]]$ into an algebra.
The action is equivalent to giving commutation rules
\[x \cdot r = r\cdot x + \sum_{k\ge 1} \phi_k(r) x^{1+k}.\]
\end{remark}

\begin{remark}
    $\Upsilon$ defines the algebra structure on $T\otimes_RT = R[[x,y]]$.
It determines, and is determined by, a commutation rule
\[y\cdot x = \sum_{i,j\ge 0} \Upsilon_{i,j}x^{1+i}y^{1+j}.\]
This follows because by definition $\Upsilon(x\otimes x)$ computes $y\cdot x$ in
$R[[x,y]]$.
The compatibility relations $(\epsilon\otimes\id)\Upsilon = (\id\otimes\epsilon)$
and $(\id\otimes\epsilon)\Upsilon = (\epsilon\otimes\id)$ imply
$(\epsilon\otimes\id) \Upsilon(x\otimes x)
= (\id\otimes\epsilon) \Upsilon(x\otimes x) = 0$
which explains why the coefficients
$\Upsilon_{-1,j}$ or $\Upsilon_{i,-1}$ are not needed.
More restrictions on the $\Upsilon_{i,j}$ follow from the commutativity assumption (see remark \ref{commexpl} below).
\end{remark}

\begin{remark}
    The Yang-Baxter equations for $\Upsilon$
arrange that $\Upsilon$ can be used to give a well-defined algebra structure
on all tensor powers
\[T\otimes_R \cdots \otimes_R T = R[[x_1,x_2,\ldots,x_n]].\]
Here for all $r\in R$, $1\le j\le n$ one has
$$x_j r = \sum_{k\ge_0} \phi_k(r) x_j^{1+k}
\qquad
\text{and for $l>k$}
\qquad
x_l x_k = \sum_{i,j\ge 0} \Upsilon_{i,j} x_k^{1+i} x_l^{1+j}.$$
\end{remark}

\begin{remark}
The coproduct $\Delta \colon{} R[[x]] \rightarrow R[[x,y]]$ is multiplicative,
hence determined by
the image of $x$ which we write as $x+_Fy$:
\[x+_F y := \Delta(x).\]
From $(\epsilon\otimes\Delta) = \id = (\id\otimes\epsilon)\Delta$ one finds
\[x+_Fy = x + y + \sum_{i,j\ge 0} a_{i,j} x^{1+i}y^{1+j}.\]
The associative law for $\Delta$ translates to $x+_F(y+_Fz) = (x+_Fy)+_Fz$
and with $[-1](x) = \chi(x)$ one has $x+_F [-1](x) = [-1]_F(x) +_F x = 0$.
\end{remark}

\begin{remark}
As in the classical commutative case the $\chi$ can be reconstructed from $\Delta$,
hence could be dropped from the data defining the formal group law.
\end{remark}

\begin{remark}
The cocommutativity $\Upsilon \Delta = \Delta$ asks that
$x+_Fy = y+_F x$. So in particular the usual multiplicative formal group law
\[x+_Fy = x + y + xy\]
is \emph{not} in general commutative, unless $xy = yx$ in $R[[x,y]]$.
\end{remark}

\begin{remark}\label{commexpl}
    The commutativity $m\Upsilon = m$ does \emph{not} require
    any of $R$ or $T$ or $T\otimes_RT$ to be commutative rings.
    Instead it asks that the specialization map
    $T\otimes_RT = R[[x,y]] \rightarrow R[[x]]$ with $f(x,y) \mapsto f(x,x)$
    be invariant under $\Upsilon$.
    For $f(x,y) = yx$, for example, this gives
    \[x^2 = \sum_{i,j\ge 0} \Upsilon_{i,j} x^{2+i+j}\]
    from which one can deduce $\Upsilon_{0,0}=1$ and
    for $k>2$
    the weak anti-symmetry condition $\sum_{i+j=k} \Upsilon_{i,j} = 0$.
\end{remark}

\begin{remark}
For any ring $R$ one can consider the trivial $\BFK$ action with
$\phi_k(r) = 0$ for all $r\in R$, $k\ge 1$. This makes $x$ central in $R[[x]]$
and one can likewise let $\Upsilon = \id$ to achieve $xy = yx$. For commutative $R$
this then recovers the usual notion of a commutative one-dimensional formal group law.
\end{remark}

\begin{remark}
Schauenburg observed in \cite{MR1656075} that in the presence of an antipode
the braiding $\Upsilon$
can be recovered from the multiplication and comultiplication:
one has
\begin{equation}\label{schauenburg}
    \Upsilon = (m\otimes m) (\chi \otimes \Delta m \otimes\chi) (\Delta\otimes\Delta).
\end{equation}
He shows that a braided commutative Hopf algebra like ours always has $\Upsilon^2 = \id$.
Schauenburg's formula (\ref{schauenburg}) can be given an amusing derivation using
formal group notation:
\[\begin{tikzcd}[column sep = small]
    {T\otimes_R T} && {(x,y)} \\
    {T\otimes_R T \otimes_R T\otimes_R T} && {(x_1+_Fx_2,y_1+_Fy_2)} \\
    {T\otimes_R T \otimes_R T} && {(x_1+_Fa,a+_Fy_2)} \\
    {T\otimes_R T \otimes_R T\otimes_R T} && {([-1](x_1)+_Fa_1+_Fa_2,a_1+_Fa_2+_F[-1](y_2))} \\
    {T\otimes_R T} && {\underset{\substack{\quad\\\quad\\=\,(y,x)}}{([-1](x)+_Fx+_Fy,x+_Fy+_F[-1](y)}}
    \arrow["\Delta\otimes\Delta"', from=1-1, to=2-1]
    \arrow["\id\otimes m\otimes\id"', from=2-1, to=3-1]
    \arrow["\chi\otimes\Delta\otimes\chi"', from=3-1, to=4-1]
    \arrow["m\otimes m"', from=4-1, to=5-1]
    \arrow[maps to, from=1-3, to=2-3]
    \arrow[maps to, from=2-3, to=3-3]
    \arrow[maps to, from=3-3, to=4-3]
    \arrow[maps to, from=4-3, to=5-3]
\end{tikzcd}\]
\end{remark}

\end{subsection}
\begin{subsection}{Triangular Hopf module algebras over $\BFK$}
In order to understand formal group laws it makes sense to 
consider the underlying braiding $\Upsilon$ separately from the formal group data $(\Delta,\chi)$.

Let $R$ be a Hopf module algebra over $\BFK$. 
\begin{definition}\label{def:triangular}
    A \emph{triangular structure on $R$} is a $\Upsilon \colon{} T\otimes_RT \rightarrow T\otimes_RT$
    with $\Upsilon^2 = \id$
    that satisfies the Yang-Baxter equation (\ref{yangbaxter})
    and commutes with $m\colon{}T\otimes_RT\rightarrow T$ and $u:R\rightarrow T$ 
    in the sense of equations (\ref{ycomm1}) and (\ref{ycomm2}).
\end{definition}
The pair $(R,\Upsilon)$ is then called a \emph{triangular $\BFK$-Hopf module algebra}.
We chose the name \enquote{triangular} 
because the Yang-Baxter equation is historically also known as the \enquote{star-triangle relation}
and as such has already inspired similar names in the Hopf algebra literature.

Let $(R,\Upsilon)$ be a triangular $\BFK$-Hopf module algebra. One has 
\[yx = \sum_{p,q} \Upsilon_{p,q} x^{1+p}y^{1+q} \equiv \psi(x) y \mod y^2\]
for some 
\[\psi(x) = \sum_{k\ge 0} \Upsilon_{k,0}x^{1+k} = x + \Upsilon_{1,0}x^2 + \Upsilon_{2,0}x^3 + \cdots.\]
The triangular structure is called \emph{strict} if one has $\psi(x)=x$ here.
The following Lemma shows that every $(R,\Upsilon)$ can be strictified by a coordinate change\footnote{%
Secretly this proof is modeled on a hypothetical representation $x\mapsto Z_0T + Z_1T^2 + Z_2T^3 + \cdots$,
similar to the one we will establish below for $M\xi$ in Lemma \ref{lem:hurlem}. 
For the strict triangular structure of $M\xi$ one has $Z_0=1$ whereas we here merely assume $Z_0$ to be invertible. 
This analogy suggests that replacing $x$, $y$ with 
$Z_0^{-1}x$, $Z_0^{-1}y$ strictifies the triangular structure.}
$x \leftrightarrow \lambda x$, possibly after adjoining a suitable $\lambda$:
\begin{lemma}
Let $R$ be a triangular $\BFK$-Hopf module algebra with braiding $yx = \sum\Upsilon_{i,j}x^{1+i}y^{1+j}$.
Then $R'=R[\lambda,\lambda^{-1}]$ admits a triangular $\BFK$-Hopf module algebra structure extending the one on $R$ with 
$\phi_k(\lambda) = \lambda\Upsilon_{k,0}$. 
In this structure one has
\begin{equation}\label{eq:lamlem}
    (\lambda y) (\lambda x) = \sum_{i,j\ge 0} \tilde \Upsilon_{i,j} (\lambda x)^{1+i} (\lambda y)^{1+j}
\end{equation}
with $\tilde \Upsilon_{k,0} = \tilde\Upsilon_{0,k} = 0$ for $k\not=0$.
\end{lemma}
\begin{proof}
One first needs to check that $\phi_k(\lambda) = \lambda\Upsilon_{k,0}$ 
does indeed define a $\BFK$-Hopf module algebra structure.
This comes down to verifying that the relation 
\begin{equation}\label{eq:lamdef} 
    x \lambda = \lambda x + \sum_{k\ge 0} \lambda \Upsilon_{k,0} x^{1+k}
\end{equation}
gives a consistent $R'$ algebra structure on $R'[[x]]$.

By assumption  \eqref{eq:lamdef} holds modulo $y^2$ in $R[[x,y]]$ with $\lambda \rightsquigarrow y$.
Let $R[[x,\bar y]]$
be the associated graded ring for the filtration by powers of the ideal $I=(y)\subset R[[x,y]]$.
Then \eqref{eq:lamdef} holds on the nose in $R[[x,\bar y]]=\left(R[[\bar y]]\right)[[x]]$ which shows that we can 
indeed set $\lambda = \bar y$ and $R'=R[\bar y, \bar y^{-1}]$.

From $\phi_k(\lambda) = \lambda\Upsilon_{k,0}$ one now gets $\psi(x) = x + \sum_{k\ge 1} \Upsilon_{k,0}x^{1+k} = \lambda^{-1} x \lambda$. 
We have $\lambda y \equiv y \lambda$ mod $y^2$, 
so
\[ \lambda x \lambda y \equiv \lambda^2 \psi(x) y \equiv \lambda^2 yx \equiv \lambda y \lambda x \mod y^2. \] 
Since $y$ and $\lambda y$ generate the same ideal this congruence also holds modulo $(\lambda y)^2$.
This proves that we have $\tilde\Upsilon_{k,0}=0$ for $k>0$ in the representation \eqref{eq:lamlem}.
Applying the symmetry map $\Upsilon:R[[x,y]] \rightarrow R[[x,y]]$ to interchange $x$ and $y$ then also 
gives $\lambda x \lambda y \equiv \lambda y \lambda x$ mod $(\lambda x)^2$, as claimed.
\end{proof}
In light of this Lemma we will only consider strict triangular structures in the following. 

\begin{remark}\label{rem:ypqviaphi}
    The commutation relation
    $yx = \sum_{p,q} \Upsilon_{p,q}x^{1+p}y^{1+q}$
    can be rewritten as
    \[yx = \sum_q \phi_q(x) y^{1+q},\qquad \phi_q(x) = \sum_p \Upsilon_{p,q}x^{1+p}.\]
    For a strict $\Upsilon$ this casts the braiding as a $\BFK$-module algebra structure 
    on $T=R[[x]]$ that extends the given one on $R$.
\end{remark}

It is easy to see that there is a universal strict triangular $\BFK$-Hopf module algebra:
for this let $\widehat \Tria$ be the free $\BFK$-Hopf module algebra on symbols $\Upsilon_{p,q}$,
$p,q\ge 1$, let $T=\widehat{\Tria}[[x]]$ and define 
$\Upsilon : T\otimes T \rightarrow T\otimes T$ via 
\[ \Upsilon \left(x\otimes x\right)  = x\otimes x + \sum_{p,q\ge 1} \Upsilon_{p,q}x^{1+p}\otimes x^{1+q}. \]
Then let $\Tria = \widehat{\Tria} / (\mathrm{relations})$ where the relations 
are those required by Definition~\ref{def:triangular}.

It is largely a matter of taste whether $\Tria$ should be considered as 
an algebra with or without a unit here. For definiteness, we will assume a unit $1\in\Tria$
and a corresponding augmentation $\epsilon:\Tria \rightarrow \ZZ$.
Its kernel $\bar\Tria$ defines a decreasing convergent filtration through its powers $\bar\Tria^n$.

By construction the strict triangular structure $(\Tria,\Upsilon)$ is clearly universal.
It is less obvious to determine the size and the structure of $\Tria$ more explicitly.
We shall show 
\begin{thm}\label{thm:triangular}
    Let $(\Tria,\Upsilon)$ be the universal strict triangular $\BFK$-Hopf module algebra.
    With $p,q\ge 1$ there are families of relations
    \begin{enumerate}
        \item \label{eq:thmtr1}
            $\Upsilon_{p,q} \equiv -\Upsilon_{q,p}, \, \Upsilon_{p,p} \equiv 0$  modulo $\bar\Tria^2$,
        \item \label{eq:thmtr2}
            $\phi_k \Upsilon_{p,q} + \phi_p \Upsilon_{q,k} + \phi_q \Upsilon_{k,p} \equiv 0$ modulo $\bar\Tria^2$,
        \item \label{eq:thmtr3}
            $\phi_p\phi_q(a) \equiv \phi_q\phi_p(a)$ modulo $\bar\Tria^{n+1}$ for all $a\in\bar\Tria^n$.
    \end{enumerate}
    As an algebra $\Tria$ is the free associative algebra on the family 
    \[ \left\{ \phi_{i_1}\cdots \phi_{i_n} \Upsilon_{p,q} \,:\, i_1\ge i_2\ge \cdots \ge i_n \ge p < q \right\}.\]
\end{thm}

\begin{proof}[Proof of Thm.~\ref{thm:triangular}, part 1: the relations (\ref{eq:thmtr1}) - (\ref{eq:thmtr3})]
The crucial observation is that from 
$yx = xy + \sum_{p,q}\Upsilon_{p,q}x^{1+p}y^{1+q}$
we get $xy \equiv xy \bmod \bar\Tria$
and hence also $x^ky^l \equiv y^l x^k \bmod {\bar \Tria}$
for all $k$, $l$. 

Using this we find
$yx \equiv xy + \sum_{p,q}\Upsilon_{p,q}y^{1+q}x^{1+p}  \pmod {\bar \Tria^2}$.
Since we also have 
$xy = yx - \sum_{p,q}\Upsilon_{q,p}y^{1+q}x^{1+p}$
a comparison of coefficients yields $\Upsilon_{p,q}\equiv - \Upsilon_{q,p}$.
Setting $x=y$ then gives $x^2 \equiv x^2 + \sum_{p} \Upsilon_{p,p}x^{2+2p}$
which proves $\Upsilon_{p,p}\equiv 0$ as well.

A similar computation established the Jacobi-like relation (\ref{thm:triangular}.\ref{eq:thmtr2}):
From $[y,z] = \sum \Upsilon_{p,q}y^{1+p}z^{1+q}$ we get 
\begin{align*} 
    [x,[y,z]] & = \sum_{k\ge 0, p,q\ge 1} \phi_k \Upsilon_{p,q}x^{1+k} y^{1+p}z^{1+q}
- \sum_{p,q\ge 1} \Upsilon_{p,q}y^{1+p}z^{1+q}x \\
& \equiv \sum_{k,p,q\ge 1}   \phi_k \Upsilon_{p,q}x^{1+k} y^{1+p}z^{1+q} 
\pmod {\bar \Tria^2 [[x,y,z]]}.
\end{align*}
Modulo $\bar\Tria$ we can rearrange the variables on the right-hand side and get 
\begin{align*}
    [x,[y,z]] & \equiv  \sum_{k,p,q\ge 1}   \phi_k \Upsilon_{p,q}y^{1+p} z^{1+q}x^{1+k} 
    \pmod {\bar \Tria^2 [[x,y,z]]}
    \\
    & \equiv   \sum_{k,p,q\ge 1}   \phi_k \Upsilon_{p,q}z^{1+q}x^{1+k}y^{1+p}
    \pmod {\bar \Tria^2 [[x,y,z]]}.
\end{align*}
We now reorder the variables in $[x,[y,z]] + [y,[z,x]] + [z,[x,y]] = 0$.
Extracting the coefficient of $x^{1+k} y^{1+p}z^{1+q}$
then gives the desired $\phi_k \Upsilon_{p,q} + \phi_p \Upsilon_{q,k} + \phi_q \Upsilon_{k,p} \equiv 0$.

For the last family of relations we first establish directly that 
\[ 
    \phi_{k_1}\cdots \phi_{k_n} \Upsilon_{p,q} \equiv 
    \phi_{\sigma(k_1)}\cdots \phi_{\sigma(k_n)} \Upsilon_{p,q} \pmod{\bar\Tria^2} 
\]
for every permutation $\sigma$ and sequence of integers $k_1,\ldots,k_n$.
This is seen by looking at 
\begin{align*}
    x_1\cdots x_n\cdot [z,y] &= \sum_{k1,\ldots,k_n\ge 1}
        \phi_{k_1}\cdots \phi_{k_n}\Upsilon_{p,q}  x_1^{1+k_1}\cdots x_n^{1+k_n}  y^{1+p}z^{1+q} 
    \\ &\equiv  \sum_{k1,\ldots,k_n\ge 1} \phi_{k_1}\cdots \phi_{k_n}\Upsilon_{p,q}
       x_{\sigma(1)}^{1+k_{\sigma(1)}}\cdots x_{\sigma(n)}^{1+k_{\sigma(n)}} y^{1+p}z^{1+q} 
\end{align*}
and comparing it against 
\[ x_{\sigma(1)}\cdots x_{\sigma(n)} \cdot [z,y] 
= \sum_{k1,\ldots,k_n\ge 1} \phi_{k_1}\cdots \phi_{k_n}\Upsilon_{p,q}
x_{\sigma(1)}^{1+k_{1}}\cdots x_{\sigma(n)}^{1+k_{n}} y^{1+p}z^{1+q}.
\]
Finally, to establish relation (\ref{eq:thmtr3}) we use that every $a\in\Tria$ 
decomposes as a sum of products of terms of the form just considered. It thus suffices to show 
$\phi_p\phi_q(ab) \equiv \phi_q\phi_p(ab) \bmod \bar\Tria^{n+2}$, 
assuming that (\ref{eq:thmtr3}) is already satisfied for $a\in\bar\Tria$, $b\in\bar\Tria^n$ separately. 
But this follows easily from the product formula 
$\phi_p\phi_q(ab) = \sum \phi_p'\phi_q'(a) \cdot \phi_p''\phi_q''(b)$.
\end{proof}

\begin{remark}\label{rem:triauniv}
With these relations we can now establish an upper bound on the size of $\Tria$:
by definition it is generated as an algebra by expressions 
$\phi_{i_1}\cdots \phi_{i_n} \Upsilon_{p,q}$.
By (\ref{eq:thmtr1}) we can assume $p<q$ here. Should we have 
$i_n<p$ we can use the Jacobi relation (\ref{eq:thmtr2}) to rewrite 
$\phi_{i_n}\Upsilon_{p,q} \leadsto \phi_p\Upsilon_{i_n,q} - \phi_q \Upsilon_{i_n,p}$
where we always have $a>b<c$ in each $\phi_a\Upsilon_{b,c}$ on the right-hand side.
Reordering the $\phi_{i_k}$ and repeating this trick will get us down 
to products of terms with $i_1\ge i_2\ge \cdots\ge i_n>p<q$.
We conclude that there is a surjective algebra map 
\[ \Ass\left( \phi_{i_1}\cdots \phi_{i_n} \Upsilon_{p,q} \,:\, i_1\ge i_2\ge \cdots \ge i_n \ge p < q \right) \twoheadrightarrow \Tria. \]
We will later (following Lemma \ref{lem:liefilt}) show, by explicit computation, that this map is actually \emph{bijective} for the 
triangular structure of $M\xi$. This will then complete the proof of Theorem \ref{thm:triangular}
while also giving us an explicit computational model for $\Tria$ as a Hopf module algebra. 
\end{remark}

\begin{remark}
Even though the relations (1)-(3) become familiar Lie algebra relations 
when reduced modulo $\bar\Tria^2$
their unreduced form is extraordinarily complicated.
We give some examples:
for the first family of relations the non-trivial low-dimensional unreduced equations are 
(writing $\phi_{k_1,\ldots,k_n}$ for $\phi_{k_1}\cdots\phi_{k_n}$)
\begin{align*}
    \Upsilon_{2,4} + \Upsilon_{4,2} =& -\Upsilon_{3,3} = 6\cdot\Upsilon_{1,2}^2 \\
    \Upsilon_{2,5} + \Upsilon_{5,2} =& -\Upsilon_{3,4} - \Upsilon_{4,3} = 6\cdot\Upsilon_{1,2}\cdot\Upsilon_{1,3} + 6\cdot\Upsilon_{1,2}\cdot\phi_{1}(\Upsilon_{1,2}) + 8\cdot\Upsilon_{1,3}\cdot\Upsilon_{1,2} \\
    \Upsilon_{3,5} + \Upsilon_{5,3} =& -6\cdot\Upsilon_{1,2}\cdot\Upsilon_{1,4} + 6\cdot\Upsilon_{1,2}\cdot\Upsilon_{2,3} + \Upsilon_{1,2}\cdot\phi_{1,1}(\Upsilon_{1,2}) - 9\cdot\Upsilon_{1,2}\cdot\phi_{2}(\Upsilon_{1,2})
    \\& - 8\cdot\Upsilon_{1,3}\cdot\phi_{1}(\Upsilon_{1,2}) - 10\cdot\Upsilon_{1,4}\cdot\Upsilon_{1,2} + 12\cdot\Upsilon_{2,3}\cdot\Upsilon_{1,2} \\
    \Upsilon_{4,4} =& -6\cdot\Upsilon_{1,2}\cdot\Upsilon_{2,3} - 6\cdot\Upsilon_{1,2}\cdot\phi_{1}(\Upsilon_{1,3}) - 3\cdot\Upsilon_{1,2}\cdot\phi_{1,1}(\Upsilon_{1,2}) 
    \\& + 3\cdot\Upsilon_{1,2}\cdot\phi_{2}(\Upsilon_{1,2}) - 8\cdot\Upsilon_{1,3}\cdot\Upsilon_{1,3} - 4\cdot\Upsilon_{1,3}\cdot\phi_{1}(\Upsilon_{1,2}) - 12\cdot\Upsilon_{2,3}\cdot\Upsilon_{1,2} \\
    \Upsilon_{4,5} + \Upsilon_{5,4} =&\hphantom{+}
    69\cdot\Upsilon_{1,2}^2\cdot\Upsilon_{1,2} - 6\cdot\Upsilon_{1,2}\cdot\Upsilon_{2,4} - 6\cdot\Upsilon_{1,2}\cdot\phi_{1}(\Upsilon_{1,4}) 
    \\& - 5\cdot\Upsilon_{1,2}\cdot\phi_{1,1}(\Upsilon_{1,3}) - \Upsilon_{1,2}\cdot\phi_{1,1,1}(\Upsilon_{1,2}) - 9\cdot\Upsilon_{1,2}\cdot\phi_{2}(\Upsilon_{1,3}) 
    \\& - 3\cdot\Upsilon_{1,2}\cdot\phi_{2,1}(\Upsilon_{1,2}) + 9\cdot\Upsilon_{1,2}\cdot\phi_{3}(\Upsilon_{1,2}) - 8\cdot\Upsilon_{1,3}\cdot\Upsilon_{1,4} - 8\cdot\Upsilon_{1,3}\cdot\Upsilon_{2,3} 
    \\& - 16\cdot\Upsilon_{1,3}\cdot\phi_{1}(\Upsilon_{1,3}) - 6\cdot\Upsilon_{1,3}\cdot\phi_{1,1}(\Upsilon_{1,2}) - 10\cdot\Upsilon_{1,4}\cdot\Upsilon_{1,3} 
    \\& - 5\cdot\Upsilon_{1,4}\cdot\phi_{1}(\Upsilon_{1,2}) - 12\cdot\Upsilon_{2,3}\cdot\Upsilon_{1,3} - 18\cdot\Upsilon_{2,3}\cdot\phi_{1}(\Upsilon_{1,2}) - 15\cdot\Upsilon_{2,4}\cdot\Upsilon_{1,2}
    \end{align*}
    For the $\Lambda_{k,p,q} = \phi_k \Upsilon_{p,q} + \phi_p \Upsilon_{q,k} + \phi_q \Upsilon_{k,p}$
    one finds $\Lambda_{1,2,3} = -\Upsilon_{1,2}^2$ and
\begin{align*}
    \Lambda_{1,2,4} =&\hphantom{+} 2\cdot\Upsilon_{1,2}\cdot\Upsilon_{1,3} - 2\cdot\Upsilon_{1,2}\cdot\phi_{1}(\Upsilon_{1,2}) - 4\cdot\Upsilon_{1,3}\cdot\Upsilon_{1,2} + 6\cdot\phi_{1}(\Upsilon_{1,2})\cdot\Upsilon_{1,2} \\
    \Lambda_{1,3,4} =&\hphantom{+} 3\cdot\Upsilon_{1,2}\cdot\Upsilon_{1,4} - 3\cdot\Upsilon_{1,2}\cdot\Upsilon_{2,3} - 3\cdot\Upsilon_{1,2}\cdot\phi_{1}(\Upsilon_{1,3}) - 6\cdot\Upsilon_{1,2}\cdot\phi_{1,1}(\Upsilon_{1,2}) 
    \\& - 6\cdot\Upsilon_{1,2}\cdot\phi_{2}(\Upsilon_{1,2}) - 2\cdot\Upsilon_{1,3}\cdot\Upsilon_{1,3} - 7\cdot\Upsilon_{1,3}\cdot\phi_{1}(\Upsilon_{1,2}) - 2\cdot\Upsilon_{1,4}\cdot\Upsilon_{1,2} 
    \\& + 3\cdot\Upsilon_{2,3}\cdot\Upsilon_{1,2} - 3\cdot\phi_{1}(\Upsilon_{1,2})\cdot\phi_{1}(\Upsilon_{1,2}) - 8\cdot\phi_{1}(\Upsilon_{1,3})\cdot\Upsilon_{1,2} + 9\cdot\phi_{2}(\Upsilon_{1,2})\cdot\Upsilon_{1,2}
    \end{align*}
    Finally, as an illustration of the last family of relations we just mention 
    \[ \left(\phi_1 \phi_3 - \phi_3 \phi_1\right)  \Upsilon_{1,4} = 
    3\cdot \Upsilon_{1,2}\cdot \phi_{1}(\Upsilon_{1,4}) + \Upsilon_{1,3}\cdot \Upsilon_{1,4}
     - \Upsilon_{1,4}\cdot\Upsilon_{1,3} - 2\cdot\phi_{1}(\Upsilon_{1,4})\cdot \Upsilon_{1,2} . \]
\end{remark}

\end{subsection}

\end{section}


\begin{section}{Symmetric and quasi-symmetric functions}
We have seen that for a ring $R$ the datum of a $\BFK$-Hopf module algebra structure on $R$ is 
exactly what is needed to make sense of the power series ring $R[[x]]$ where $x$ is allowed to be non-central.
An additional triangular structure gives meaning to $R[[x_1,\ldots,x_n]]$ for any number $n$ of variables.
We next investigate the canonically defined subrings $\Sym_n(R) \subset \QSym_n(R) \subset R[[x_1,\ldots,x_n]]$
of symmetric or quasi-symmetric functions.
\begin{subsection}{Symmetric functions}\label{sec:symm}
Since $\Upsilon^2=\id$ we have an action of the
symmetric group $\Sigma_n$ on the tensor powers of $T$.
The invariants define the rings of symmetric functions
$\Sym_n(R) \subset R[[x_1,\ldots,x_n]]$.
Note again that the $x_j$
are neither assumed to commute with the coefficients from $R$ nor with each other.

Our goal is to show that the $\Sym_n(R)$ are generated by
symmetric functions $c_1,\ldots,c_n \in T\otimes_R\cdots\otimes_RT = R[[x_1,\ldots,x_n]]$
and that $\Sym_n(R)$ is the free $R$-module with basis given by the monomials $c_1^{e_1}\cdots c_n^{e_n}$.
In short, and in deceptively familiar notation, we want to show that
\[\Sym_n(R) = R[[c_1,\ldots,c_n]]\]
with the understanding that
\begin{enumerate}
    \item The $c_k$ do not necessarily commute with the coefficients from $R$.
    \item They do not necessarily commute with each other.
    \item They are \emph{not} defined canonically: choices are involved in their  construction.
\end{enumerate}
We furthermore want to show that the $c_k$ can be chosen compatibly for different $n$, so that
$c_k\in \Sym_n(R)$ can be obtained by restriction from some $c_k\in\Sym_\infty(R)$ with
\[\Sym_\infty(R) = R[[c_1,c_2,\ldots]].\]

We start with an ad-hoc computation of $\Sym_2(R) \subset R[[x,y]]$
before dealing with the general case systematically below.
\begin{lemma}
One has $yx = xy + \sum_{p<q} \theta_{p,q} \left(x^py^q - y^px^q\right)$
for certain uniquely determined $\theta_{p,q}\in R$.
\end{lemma}
\begin{proof}
    One clearly has $yx = xy$ + terms of degree $\ge 3$. It follows that the $\beta_{p,q}$
    with
    \[\beta_{p,q} = \begin{cases}x^py^q & p<q \\ y^px^q & p > q \\ x^py^p & p=q \end{cases}\]
    form a left $R$-module basis of $R[[x,y]]$.
    We therefore have
    \begin{equation}\label{eq:thetadef}
        yx = \sum_{p,q\ge 1} \theta_{p,q} \beta_{p,q}
    \end{equation}
    with unique $\theta_{p,q}\in R$.
    We claim that $\theta_{p,q} = -\theta_{q,p}$ for $p\not=q$,
    $\theta_{1,1} = 1$ and $\theta_{p,p} = 0$ for $p>1$.

    The commutativity of $m\colon{}T\otimes_RT\rightarrow T$ gives
    \[x^2 = \sum_{p,q\ge 1} \theta_{p,q} x^{p+q}.\]
    From this one gets $\theta_{1,1}=1$.
    Applying $\Upsilon$ to (\ref{eq:thetadef})
    and using $\Upsilon(\beta_{p,q}) = \beta_{q,p}$ for $p\not=q$ gives
    \begin{align*}
        xy &= \Upsilon(xy) + \sum_{p\not=q} \theta_{p,q} \beta_{q,p} + \sum_{p\ge 2} \theta_{p,p} \Upsilon(x^py^p)
        \\ &= xy + \sum_{p\not=q} (\theta_{p,q} + \theta_{q,p})\beta_{p,q}
        + \sum_{p\ge 2} \theta_{p,p} \Upsilon(x^py^p)
    \end{align*}
    If one assumes by induction that $\theta_{2,2} = \cdots = \theta_{k,k} = 0$
    one can compare coefficients up to degrees less that $2(k+1)$ and
    obtain $\theta_{p,q}  = -\theta_{q,p}$ in this range.
    Specializing to $y=x$ then gives $\theta_{k+1,k+1}=0$, completing the induction.
\end{proof}
\begin{corollary}
Let $e_2(x,y) = xy + \sum_{p<q} \theta_{p,q} x^py^q$.
Then $e_2\in \Sym_2(R)$.
\end{corollary}
\begin{proof}
    One has
    \begin{align*}
        \Upsilon(e_2) &= yx + \sum_{p<q} \theta_{p,q} \Upsilon(x^px^q)
        \\&= xy + \sum_{p<q} \theta_{p,q}\left( x^py^q-y^px^q +y^px^q \right) = e_2.
    \end{align*}
\end{proof}
One can now check that with $e_1=x+y$ one has $\Sym_2(R) = R[[e_1,e_2]]$
as expected.

To define general $c_k$, though, we need to use the theory of quasi-determinants
and noncommutative symmetric functions.
The paper \cite{MR1398918} is a nice introduction to this circle of ideas
(which are also covered in more depth in \cite{MR2132761}*{sect.~6.5}).

Given any collection of (non-commuting) variables $x_1,\ldots,x_n$
one tries to think of these as solutions to a common equation
\[x^n + c_1 x^{n-1} + c_2 x^{n-2} + \cdots + c_n = 0.\]
Assuming the $x_j$ to be generic enough, they will determine this equation,
and hence also the coefficients $c_k$, uniquely.
The coefficients $c_k$ are then functions of the $x_1,\ldots,x_n$,
and by construction obviously independent of the order of the $x_j$.

To compute such $c_k$ explicitly one uses the ansatz
\begin{equation}\label{ckelem}
    c_k(x_1,\ldots,x_n) = (-1)^k \sum_{i_1<i_2<\cdots<i_k} y_{i_k} y_{i_{k-1}}\cdots y_{i_2}y_{i_1}.
\end{equation}
That is one lets $c_k$ be the $k$-th elementary symmetric function
of yet to be determined new variables $y_1,y_2,\ldots,y_n$.
Note that the order of the factors on the right-hand side is significant since the
$y_j$ will also not commute with each other.

The main theorem of \cite{MR2132761}*{sect.~6.5} is that this ansatz
can be validated using quasi-determinants $v_k = v_k(x_1,\ldots,x_k)$
of a certain Vandermonde matrix: the $y_j$ can then be computed as
$y_j = v_j x_j v_j^{-1}$.

We follow this approach to define the required symmetric
$c_1, c_2, \ldots \in \Sym_\infty(R)$
as the coefficients in the doubly-infinite system of equations
\begingroup
\arraycolsep=3.4pt\def\arraystretch{1.2}
\begin{equation}\label{eqtower}
    \begin{array}{cccccccccccc}
        1 &+& c_1x_1^{-1} &+& c_2x_1^{-2} &+& c_3x_1^{-3} &+& \cdots &=& 0\\
        1 &+& c_1x_2^{-1} &+& c_2x_2^{-2} &+& c_3x_2^{-3} &+& \cdots &=& 0\\
        1 &+& c_1x_3^{-1} &+& c_2x_3^{-2} &+& c_3x_3^{-3} &+& \cdots &=& 0\\
        \vdots && \vdots && \vdots && \vdots && \vdots &\vdots& \vdots
    \end{array}
\end{equation}
\endgroup
We need to solve this recursively for the $y_j$ that appear in the ansatz (\ref{ckelem}).
We will inductively assume that $y_1,\ldots,y_{k-1}$ are already known and work modulo
$y_{k+1},y_{k+2},\ldots$ to determine $y_k$ from the $k$-th equation in (\ref{eqtower}).

Note that modulo $y_{k+1},y_{k+2},\ldots$ we have $c_{k+1} \equiv c_{k+2} \equiv \cdots \equiv 0$
so the $k$-th equation in (\ref{eqtower}) reduces to
\begin{equation}\label{kthrow}
    x_k^k + c_1x_k^{k-1} + c_2x_k^{k-2} + \cdots + c_k = 0.
\end{equation}
For $k=1$ this gives $y_1 = x_1$. For $k=2$ we are looking at
\[x_2^2 + c_1x_2 + c_2 = 0 \quad \text{with $c_1=-y_1-y_2$, $c_2=y_2y_1$, $y_1=x_1$}\]
which gives \[(x_2-x_1)x_2 = y_2(x_2-x_1).\]
We deal with the general case in the following Lemma.
\begin{lemma}\label{lem:symm1}
    There are $y_k$ in $R[[x_1,\ldots,x_k]]$
    such that (\ref{ckelem}) defines a family of symmetric functions $c_k$ with
    $\Sym_\infty(R)=R[[c_1,c_2,\ldots]]$.
\end{lemma}
We will refer to these $y_k$ as \emph{Vieta coordinates} in the sequel.
\begin{proof}
We follow the inductive approach outlined above.
For the inductive step assume that $y_1,\ldots,y_{k-1}$ have been determined
and that $y_j\equiv x_j$ mod degrees $\ge 2$.
Working modulo $y_{k+1},y_{k+2},\ldots$ we find
\begin{align*}
    c_l &= (-1)^l \sum_{i_1<i_2<\cdots<i_l} y_{i_l} \cdots y_{i_2}y_{i_1}
\\ &= (-1)^l \sum_{i_1<i_2<\cdots<i_l = k} y_k y_{i_{l-1}} \cdots y_{i_2}y_{i_1}
    + (-1)^l \sum_{i_1<i_2<\cdots<i_l < k} y_{i_l} \cdots y_{i_2}y_{i_1}
\\&= -y_k d_{l-1} + d_l
\end{align*}
where $d_l = c_l(x_1,\ldots,x_{k-1})$ is already known.
Furthermore $d_k=0$, so (\ref{kthrow}) can be written as
\begin{align*}
    v_{k}(x_k) x_k &= y_k v_{k}(x_k)
\end{align*}
where $v_{k}(x) = x^{k-1} + d_1 x^{k-2} + \cdots + d_{k-1}$ is a monic polynomial of degree $k-1$.

We now construct $y_k$ as the limit of $y_k^{(d)}$ where $y_k^{(1)} = x_k$
and $y_k^{(d)} = y_k^{(d-1)} + $ correction terms of degree $d$.
We inductively assume given $y_k^{(d-1)}$ with $v_k x_k \equiv y_k^{(d-1)}v_k$ mod degrees $\ge d$,
so the difference $\delta_d = v_k x_k - y_k^{(d-1)}v_k$ consists of terms of degree $d$ or higher.
We extract the terms $\overline{\delta_d}$ of degree exactly $d$ from $\delta_d$ and perform a division with remainder
\[\overline{\delta_d} = z_d \cdot v_k + r_d.\]
We use the fact that commutators in $R[[x_1,\ldots,x_n]]$ raise degrees by at least one, so modulo degree $d+1$ we can treat
the $x_i$ as central variables that commute with both the other $x_j$ and the coefficients from $R$.
We can then argue in the familiar way that $r_d$ is a polynomial in $x_k$ of degree less than $k$ that vanishes
at the $k$-points $x_1,\ldots,x_k$, so must be identically zero.
This proves that modulo terms of degree larger than $d$ the $\overline{\delta_d}$ is cleanly
divisible by $v_k$ and we can let $y_k^{(d+1)} = y_k^{(d)} + z_d$.

This proves that the ansatz (\ref{ckelem}) defines the required $c_k\in\Sym_\infty(R)$.
Now use that the associated graded of the degree filtration of $R[[x_1,\ldots,x_n]]$ is a
power series ring $R[[\overline{x_1},\ldots,\overline{x_n}]]$ with each $\overline{x_j}$ central.
For this ring the classical theory shows that the $\Sigma_n$ invariants are
$R[[\overline{c_1},\overline{c_2},\ldots]]$ where $\overline{c_k}$ is represented by $c_k$.
An induction over the degree filtration hence gives the required $\Sym_\infty(R) = R[[c_1,c_2,\ldots]]$.
\end{proof}

The reader can find some explicit computations of the Vieta coordinates $y_k$ in section \ref{sec:chern}.
These computations take place in the Hurewicz embedding of the universal triangular structure
where a direct non-inductive computation via quasi-determinants is possible.

\end{subsection}
\begin{subsection}{Quasisymmetric functions}
Recall that a function $f(x_1,x_2,\ldots)$
is called
\emph{quasi-symmetric} if it is invariant under insertion of zero-arguments at any place:
\[f(x_1,x_2,\ldots) = f(x_1,\ldots,x_i,0,x_{i+1},x_{i+2},\ldots) \qquad\text{for all $i$.}\]
Examples of such functions are the quasi-symmetric monomials
\[m_I = \sum_{k_1<\cdots<k_n} x_{k_1}^{i_1} \cdots x_{k_n}^{i_n},
\qquad
\text{$I$ a composition $I=(i_1,\ldots i_n)$}.\]
We let $\QSym(R)$ be the free left $R$-module generated by these $m_I$.
This inherits a $\BFK$-module algebra structure by thinking of it formally as a subalgebra of the limit
$R[[x_1,x_2,\ldots]]$.
For example, recall from the proof of Lemma \ref{lem:first} that in $R[[x]]$ we have
\[x^n \cdot r = \sum_{j\ge 0} Q^{(n)}_j(\phi)(r) x^{n+j}.\]
For $\QSym(R)$ this gives the commutation rule
\begin{align*}
    m_I\cdot r &= \sum_{j_1,\ldots,j_n\ge 0} \qquad
    \sum_{k_1<\cdots<k_n} \left( Q^{(i_1)}_{j_1}(\phi) \cdots Q^{(i_n)}_{j_n}(\phi)\right) (r)\,
    x_{k_1}^{i_1+j_1} \cdots x_{k_n}^{i_n+j_n}
    \\ &= \sum_{J\ge 0} Q^{(I)}_J(\phi)(r) m_{I+J}.
\end{align*}
Coming up with a clean formula for a product $m_I\cdot m_J$ seems not so easy, but
the multiplication in $\QSym(R)$ is nonetheless straightforward to compute.
From $yx = xy + \sum_{p,q} \Upsilon_{p,q} x^{1+p}y^{1+q}$,
for example, we get
\begin{align*}
    m_1\cdot m_1 &= \left(\sum_i x_i \right) \cdot \left(\sum_j x_j \right)
    \\ &= \sum_i x_i^2 + \sum_{i<j} \left(x_ix_j + x_jx_i\right)
    \\ &= \sum_i x_i^2 + 2 \sum_{i<j} x_ix_j + \sum_{p,q} \Upsilon_{p,q} \sum_{i<j}x_i^{1+p}x_j^{1+q}
    \\ &= m_2 + 2 m_{1,1} + \sum_{p,q} \Upsilon_{p,q} m_{1+p,1+q}.
\end{align*}

\begin{lemma}
    Symmetric functions are quasi-symmetric: there is a canonical inclusion $\Sym_\infty(R)\subset\QSym(R)$.
\end{lemma}
The proof might not be fully obvious since the $c_k\in\Sym_\infty(R)$ that we defined in (\ref{ckelem})
do not immediately appear to be quasi-symmetric.
\begin{proof}
    Let $c_n=y_ny_{n-1}\cdots y_1\in\Sym_n(R)$ be the generator from Lemma \ref{lem:symm1}.
    Propagate this to
    $c'_n\in\Sym_\infty(R)$ via
    \[c'_n(x_1,x_2,\ldots) = \sum_{i_1<\cdots<i_n} c_n(x_{i_1},\ldots,x_{i_n}).\]
    Modulo terms of degree larger than $n$ the $c'_n$ agrees with the usual $n$-th order
    elementary symmetric function, so the $c'_n$ give an alternative multiplicative basis
    \[\Sym_\infty(R) = R[[c'_1,c'_2,\ldots]].\]
    Since the $c'_n$ are manifestly quasi-symmetric one has $\Sym_\infty(R)\subset \QSym(R)$.
\end{proof}
\end{subsection}
\end{section}

\begin{section}{Complex oriented ring spectra}
Let $E^\ast(-)$ be a multiplicative cohomology theory
represented by an associative (up to homotopy) ring spectrum $E$.
A \emph{complex orientation} of $E$ is the choice of a class $x\in E^2 \CP^\infty$
such that $E^\ast \CP^\infty = E^\ast[[x]]$ (as a left $E^\ast$-module)
and such that the composite
$S^0\cong \Sigma^{-2}\CP^1 \subset \Sigma^{-2}\CP^\infty \xrightarrow{x} E$
is the unit of $E$.
$E$ is called \emph{complex oriented} if such an orientation has been chosen.

The notion of a complex oriented cohomology theory has, of course,
been part of the standard repertoire of algebraic topology since the 1960s
and there are many textbook accounts for this material.
Unfortunately, most accounts explicitly or implicitly assume
that the multiplication on $E$ is homotopy-commutative.
We will not duplicate the derivation of the standard properties of such $E$
here but merely indicate the results that
remain true in the non-commutative case.

\begin{thm}
    Let $E$ be a complex oriented ring spectrum.
    Then $E^\ast$ naturally carries a one-dimensional commutative formal group law
    that corresponds to the tensor product of line bundles.
\end{thm}
\begin{proof}
To prove this we first have to exhibit a natural action of the Brouder--Frabetti--Krattenthaler
Hopf algebra $\BFK$ on $E^\ast$, then explain how to get the required structure maps
$\Delta$, $\Upsilon$, and $\chi$.
This is completely straightforward:

By assumption we have $E^\ast \CP^\infty = E^\ast[[x]]$.
This comes equipped with an $E^\ast$-algebra structure, so we have commutation rules
\[x\cdot r = \sum_{k\ge 0} \phi_k(r) x^{1+k}.\]
Here $\phi_0(r) = r$ because the $E^\ast$-bilinear suspension isomorphism 
$\tilde E^2(\CP^1) \cong \tilde E^0(S^0)$ maps $\bar x\in \tilde E^2(\CP^1)$ to $1$, so $\bar x\cdot r \leftrightarrow 1\cdot r =r \cdot 1 \leftrightarrow r \cdot \bar x$.
These $\phi_k$ define the $\BFK$-action on $E^\ast$
and we have $T = E^\ast \# S = E^\ast \CP^\infty$.

We next look at the external multiplication map
\[E^\ast \CP^\infty \otimes_{E^\ast} E^\ast\CP^\infty \rightarrow E^\ast\left(\CP^\infty\times\CP^\infty\right)\]
which is known to be an isomorphism. The flip
$(x,y) \mapsto (y,x)$ on $\CP^\infty\times\CP^\infty$
defines a bilinear $\Upsilon \colon{}  T\otimes_{E^\ast} T \rightarrow T\otimes_{E^\ast} T$,
the tensor product map $\otimes \colon{} \CP^\infty \times \CP^\infty \rightarrow \CP^\infty$
induces the required $\Delta \colon{} T \rightarrow T\otimes_{E^\ast} T$
and the inversion $\mathop{inv} \colon{} \CP^\infty \rightarrow \CP^\infty$ gives
the required $\chi \colon{} T \rightarrow T$.
\end{proof}

This result is classical for homotopy commutative ring spectra:
it is the foundational observation (due to Quillen \cite{MR253350})
for what has come to be known as
chromatic homotopy theory. In the commutative case it is well known that there
is one complex oriented ring spectrum, $MU$, that carries the
universal (for commutative rings)
one-dimensional commutative formal group law.
Furthermore, ring spectrum maps $MU\rightarrow E$ correspond one-to-one
to complex orientations on $E$
and the full moduli theory of classical formal group laws over commutative rings
has a topological stand-in in the form of the
algebra of operations and cooperations in $MU$.
We here wish to obtain a similar result for complex oriented ring spectra that are
not assumed to be homotopy commutative. The role of $MU$ will be played
by the Baker-Richter spectrum $M\xi$ from \cite{MR2484353}, \cite{MR3243390}.

\begin{lemma}
The formal group law associated to a complex oriented ring spectrum $E$ is strict.    
\end{lemma}
\begin{proof}Let $B$ be any space such that $E^\ast B$ is a free $E^\ast$ right module and consider the mapping spectrum $F=\Map(B_+,E)$. This is a ring spectrum
with $\tilde F^\ast X = \tilde E^\ast(B_+\land X) = \tilde E^\ast(B_+)\otimes_{E^\ast}\tilde E^\ast(X)$.
The map 
\[B_+\land\Sigma^{-2}\CP^\infty \xrightarrow{\mathrm{const}_+\land\id} S^0 \land \CP^\infty \xrightarrow{1_E\land x} E\land E\rightarrow E\] defines a class $y\in F^2(\CP^\infty)$
that is easily seen to be a complex orientation. 

Now let $B=\CP^\infty$. By the preceding Theorem we know that $y\cdot r \equiv r\cdot y \mod y^2$ in 
$F^\ast(\CP^\infty)$ for all $r\in F^\ast$. For $r=x\in E^2(\CP^\infty)$ this gives $yx \equiv xy \mod y^2$ in $E^\ast(\CP^\infty\times \CP^\infty)$ as claimed.
\end{proof}

\begin{subsection}{The complex oriented cohomology of $BU$ and $\Omega\Sigma\CP^\infty$}
Let $E$ be a complex oriented ring spectrum.
\begin{lemma}
    One has natural identifications
    \begin{enumerate}
        \item $E^\ast BU(n) = \Sym_n(E_\ast)$, $E^\ast BU = \Sym_\infty(E_\ast)$,
        \item $E^\ast \Omega\Sigma\CP^\infty = \QSym(E^\ast)$.
    \end{enumerate} 
\end{lemma}
\begin{proof}
The main point is to prove that the Atiyah-Hirzebruch spectral sequences $H^\ast(X;E^\ast) \Rightarrow E^\ast X$ 
for the involved spaces $X$ collapse. 
The best reference for this still seems to be \cite{MR402720}*{II}.
We leave it to the reader to verify that those arguments still work without the implicit assumption of commutativity.
\end{proof}

We learn from this that even for non-commutative $E$ one has a
ring of Chern classes $E^\ast BU(n) = E^\ast [[c_1,c_2,\ldots,c_n]]$,
but with the caveat that these $c_k$ have \emph{no canonical definition}
and hence also none of the usual nice properties.
Their construction depends on a choice of the Vieta coordinates $y_k$ as in lemma \ref{lem:symm1}
and these depend on an ordering of the variables $x_j$ in an embedding 
\[ E^\ast BU(n) \hookrightarrow E^\ast [[x_1,\ldots,x_n]]. \] 
So, for example, for these Chern classes in $E$-theory there will in general be no Whitney sum formula
due to the lack of canonicity in their definition.
And of course the $c_k$ will neither commute among themselves, nor will they commute with the scalars from $E^\ast$.

This is to be contrasted with the ring $E^\ast \Omega\Sigma\CP^\infty = \QSym(E^\ast)$: 
the complex orientation $x\in E^\ast \CP^\infty$ canonically defines the basis 
\[ m_I = \sum_{a_1<a_2<\cdots<a_k} x_{a_1}^{i_1}\dots x_{a_k}^{i_k} \in \QSym(E^\ast) \]
of monomial quasi-symmetric functions. These $m_I$ therefore define characteristic classes for  
bundles with a classifying map that is factored through $\Omega\Sigma\CP^\infty \rightarrow BU$.
If $V$, $W$ are two such bundles one finds the Whitney sum rule
\begin{equation}\label{eq:whitney}
     m_I(V\oplus W) = \sum_{I=P\amalg Q} m_P(V) \cdot m_Q(W)
\end{equation}
where $m_I \mapsto \sum_{I=P\amalg Q} m_p \otimes m_Q$ is the usual shuffle coproduct 
on quasi-symmetric functions.

For future reference let us record a few more details concerning the identification $E^\ast \Omega\Sigma\CP^\infty = \QSym(E^\ast).$
Let
\[\CC P^\infty = J_1 \rightarrow \cdots \rightarrow J_n \rightarrow J_{n+1} \rightarrow \cdots \rightarrow \Omega \Sigma \CC P^\infty\]
be the James construction on $\CC P^\infty$.
Since $J_n$ is an identification space of the $n$-fold product $\CP^\infty\times\cdots\CP^\infty$ we have 
a map $E^\ast J_n \rightarrow E^\ast[[x_1,\ldots,x_n]]$.
This identifies $E^\ast J_n$ with the module $\QSym_n(E^\ast) = E^\ast\{m_I\,:\,\len(I)\le n\}$ of quasi-symmetric functions 
of length not bigger than $n$.
The maps $\CP^\infty\times\cdots\CP^\infty \rightarrow J_n\rightarrow BU(n)$ realize the inclusions
 $\Sym_n(E^\ast) \subset \QSym_n(E^\ast) \subset E^\ast[[x_1,\ldots,x_n]]$.

\end{subsection}
    
\begin{subsection}{$M\xi$ is the universal complex oriented ring spectrum}
Now let $MU$, resp.~$M\xi$ be the Thom spectra over $BU$, resp.~$\Omega\Sigma\CP^\infty\rightarrow BU$ (see \cite{MR2484353} for the latter). 
If $E$ is complex oriented we have Thom isomorphisms $E^\ast MU \cong E^\ast BU$ and $E^\ast M\xi \cong E^\ast \Omega\Sigma\CP^\infty$.
So the computations of the last section give us a good hold on maps $MU\rightarrow E$ and $M\xi \rightarrow E$. 
We shall show 
\begin{thm}\label{thm:mxiunivor}
Let $E$ be a complex oriented ring spectrum. Then there is a canonically defined multiplicative map $M\xi \rightarrow E$
that maps the complex orientation $x^{M\xi}\in M\xi^\ast\CP^\infty$ to the orientation $x^E\in E^\ast\CP^\infty$. 
\end{thm}

To get a computational grasp on the cohomology of these Thom spectra we look at the Thom spaces 
\[ \begin{tikzcd}[row sep = small]
    M\eta \land \cdots \land M\eta \ar[r] & M\xi(n) \ar[r] & MU(n) \\
    \CP^\infty \times\cdots\times\CP^\infty \ar[r] & J_n \ar[r] & BU(n)  
\end{tikzcd} \]
Here $\eta$ is the canonical line bundle on $\CP^\infty$ and one has $E^\ast M\eta = x E^\ast[[x]]$.
\begin{lemma}
This gives a realization of $E^\ast M\xi(n)$ and $E^\ast MU(n)$ as submodules of $x_1\cdots x_n \cdot E^\ast[[x_1,\ldots,x_n]]$:
\begin{align*}
    E^\ast M\xi(n) &= x_n x_{n-1}\cdots x_1 \cdot \QSym_n(E^\ast),  \\
    E^\ast MU(n) &= y_n y_{n-1}\cdots y_1 \cdot\Sym_n(E^\ast).
\end{align*}
Here $y_1,\ldots,y_n$ are Vieta-coordinates as in lemma \ref{lem:symm1}
and the product $y_n y_{n-1}\cdots y_1$ (with this ordering of the factors) defines the Chern class $c_n\in E^\ast BU(n)$.
\end{lemma}
\begin{proof}
For $E^\ast MU(n)$ one uses the cofiber sequence $BU(n-1) \rightarrow BU(n) \rightarrow BU(n)/BU(n-1) \simeq MU(n)$. In cohomology 
this induces a short exact sequence 
\[ E^\ast MU(n) \rightarrow \Sym_n(E^\ast) \rightarrow \Sym_{n-1}(E^\ast) \] 
where the right-hand map is induced by $x_k\mapsto x_k$ for $k<n$ and $x_n \mapsto 0$,
or equivalently (since $y_k \equiv x_k$ modulo $x_1,\ldots,x_{k-1}$) $y_k \mapsto y_k$ and $y_n \mapsto 0$.
For the elementary symmetric functions $c_k$ in the $y_k$ this also gives $c_k \mapsto c_k$ and $c_n \mapsto 0$ so that the kernel 
is $E^\ast MU(n) = c_n \cdot \Sym_n(E^\ast)$ as claimed.
We leave the analogous computation for $E^\ast M\xi(n)$ to the reader.
\end{proof}

The lemma shows that we get well-defined maps 
\[ \begin{tikzcd}[row sep = small, column sep = 3cm]
    \tau^E : M\xi \ar[r, "\cdots x_n x_{n-1} \cdots x_2 x_1"] & E, \\
    \sigma^E : MU \ar[r, "\cdots y_n y_{n-1} \cdots y_2 y_1"] & E.
\end{tikzcd} \]
We can take this $\tau^E$ as the comparison map in theorem \ref{thm:mxiunivor}:
\begin{proof}[Proof of theorem \ref{thm:mxiunivor}]
    We have exhibited the map $\tau^E$ that is canonically defined from just the complex orientation
    of $E$. It remains to show that $\tau^E$ is multiplicative. This follows from the Whitney sum formula (\ref{eq:whitney}). 
\end{proof}
The map $\sigma^E : MU \rightarrow E$ on the other hand is not going to be multiplicative
in general. It is still useful: among other things it will help to construct 
a splitting of $M\xi$ in section \ref{sec:splitting}.

\begin{lemma} The map $\sigma^{M\xi} : MU \rightarrow M\xi$ is a section of $\tau^{MU} : M\xi \rightarrow MU$.
\end{lemma}
\begin{proof}
Let $y_1,y_2,\ldots,y_n$ be the Vieta coordinates in $M\xi^\ast[[x_1,\ldots,x_n]]$.
As noted in \cite{MR2132761}*{sect.~6.5} the $y_k$ can be constructed as $y_k=v_kx_kv_k^{-1}$ 
for certain Vandermonde determinants $v_k$. 
This shows that $\tau^{MU}(y_k) = x_k$ because $MU$ is commutative.
It follows that the Thom class $y_n\cdots y_1\in M\xi^\ast MU(n)$ maps to the usual Thom class $x_n\cdots x_1 = c_n\in MU^\ast MU(n)$.
In the limit this gives $\tau^{MU}\sigma^{M\xi} = \id$.
\end{proof}
\end{subsection}
\end{section}
\begin{section}{The formal group of $M\xi$}

We now want to work out how to do actual computations with the formal group law of 
$M\xi$. This serves as a preparation for the next section where we will show that 
this is in fact the universal one-dimensional commutative formal group law.

\begin{subsection}{The Hurewicz embedding}
The computations will use the Hurewicz embedding $\pi_\ast M\xi \subset H_\ast M\xi$. 
As in the classical commutative case the formal group law admits a logarithmic description
after extending the scalars from $\pi_\ast M\xi$ 
to $H_\ast M\xi \cong H_\ast \Omega\Sigma \CC P^\infty$.
We use the Hurewicz map in the form (taken from \cite{morava2020renormalization})
\[M\xi^\ast X \xrightarrow{\hur} \Hom\left(H^\ast \Omega\Sigma \CC P^\infty, H^\ast X\right) \cong H_\ast \Omega\Sigma \CC P^\infty \otimes H^\ast X.\]
Here $H=H\ZZ$ is the integral Eilenberg-MacLane spectrum and $\tau\colon{}M\xi\rightarrow H$ the natural truncation. 
For $\alpha\colon{}X\rightarrow M\xi$ and $\beta:\left(\Omega\Sigma \CC P^\infty\right)_+\rightarrow H$
the map $\hur(\alpha)(\beta)$ is given by the composition
\[X \xrightarrow{\,\alpha\,} M\xi \xrightarrow{\text{Thom diag.}}
M\xi \land \left(\Omega\Sigma \CC P^\infty\right)_+
\xrightarrow{\tau\land\beta}
H\land H \xrightarrow{\text{mult}} H.
\]
\begin{lemma}
    The map $\hur$ is multiplicative.
\end{lemma}
We leave the proof as an exercise.

Now recall that $H_\ast \Omega\Sigma\CP^\infty = \ZZ\langle Z_1,Z_2,\ldots\rangle$
is the free associative algebra on generators $Z_k$, with $Z_k$ the image of the generator 
$\beta_{k}\in H_{2k}\CP^\infty$ (and $Z_0=1$ by convention).
We look at the Hurewicz map for $X$ an $n$-fold product 
of complex projective spaces.
\begin{lemma}\label{lem:hurlem}
    The Hurewicz map 
    \[M\xi^\ast \, \left(\CP^\infty\times\cdots\times\CP^\infty\right) \xrightarrow{\quad\hur\quad} H_\ast \Omega\Sigma\CP^\infty \otimes \ZZ[[T_1,\ldots,T_n]]\]
    is injective. On the generators $x_j$ it is given by 
    \begin{equation}\label{hurxj}
        x_j \mapsto T_j + Z_1T_j^2 + Z_2T_j^3 + \cdots.
    \end{equation}
\end{lemma}
\begin{proof}
    The case of a single factor has been worked out in \cite{MR3243390}*{section 2}, see (2.4) therein and the subsequent discussion (note that our $(\hur,x,T)$ are denoted $(\Theta,x_\xi,x_H)$ there). 
    The extension to more than one factor is straightforward.
    The injectivity is proved in \cite{MR3243390}*{Prop.~2.3}.
\end{proof}
The power of this result stems from the fact that the $T_j$ on the destination 
side are \emph{central} elements. 
As a first application we use this map to compute the $\BFK$-action 
on $M\xi_\ast$.
Mapping the defining relation $x\cdot a = \sum_{k\ge 0}\phi_k(a)x^{1+k}$ 
from $M\xi^\ast[[x]]$ to $H_\ast M\xi[[T]]$ gives 
\begin{equation}\label{phikZ}
\sum_{j\ge 0} (Z_ja)\, T^{1+j} = \sum_{k\ge 0} \phi_k(a) \left(\sum_{j\ge 0} Z_j T^{1+j}\right)^{1+k}.
\end{equation}
From this the $\phi_k(a)$ can be computed recursively (see figure \ref{figphikz}).

\begin{figure}
\begin{align*} 
    \phi_{1}(Z_{2}) &= Z_{1} Z_{2} - Z_{2} Z_{1}
    &
    \phi_{2}(Z_{1}) &= -Z_{1} Z_{2} + Z_{2} Z_{1}
    \\
    \phi_{1}(Z_{3}) &= Z_{1} Z_{3} - Z_{3} Z_{1}
    &
    \phi_{2}(Z_{2}) &= -2 Z_{1} Z_{2} Z_{1} + 2 Z_{2} Z_{1}^{2}
    \\
    \phi_{3}(Z_{1}) &= -Z_{1} Z_{3} + Z_{3} Z_{1} + 3 Z_{1} Z_{2} Z_{1} - 3 Z_{2} Z_{1}^{2}
    &
    \phi_{1}(Z_{4}) &= Z_{1} Z_{4} - Z_{4} Z_{1}
\end{align*} 
\begin{align*}
    \phi_{2}(Z_{3}) &= Z_{2} Z_{3} - Z_{3} Z_{2} - 2 Z_{1} Z_{3} Z_{1} + 2 Z_{3} Z_{1}^{2}
    \\
    \phi_{3}(Z_{2}) &= -Z_{2} Z_{3} + Z_{3} Z_{2} - 2 Z_{1} Z_{2}^{2} + 2 Z_{2} Z_{1} Z_{2} + 5 Z_{1} Z_{2} Z_{1}^{2} - 5 Z_{2} Z_{1}^{3}
    \\
    \phi_{4}(Z_{1}) &= -Z_{1} Z_{4} + Z_{4} Z_{1} + 3 Z_{1} Z_{2}^{2} + 4 Z_{1} Z_{3} Z_{1} - 3 Z_{2} Z_{1} Z_{2} - 4 Z_{3} Z_{1}^{2}
    \\&\qquad  - 9 Z_{1} Z_{2} Z_{1}^{2} + 9 Z_{2} Z_{1}^{3}
    \\
    \phi_{1}(Z_{5}) &= Z_{1} Z_{5} - Z_{5} Z_{1}
    \\
    \phi_{2}(Z_{4}) &= Z_{2} Z_{4} - Z_{4} Z_{2} - 2 Z_{1} Z_{4} Z_{1} + 2 Z_{4} Z_{1}^{2}
    \\
    \phi_{3}(Z_{3}) &= -2 Z_{1} Z_{3} Z_{2} - 3 Z_{2} Z_{3} Z_{1} + 2 Z_{3} Z_{1} Z_{2} + 3 Z_{3} Z_{2} Z_{1} + 5 Z_{1} Z_{3} Z_{1}^{2}  - 5 Z_{3} Z_{1}^{3}
    \\
    \phi_{4}(Z_{2}) &= -Z_{2} Z_{4} + Z_{4} Z_{2} - 2 Z_{1} Z_{2} Z_{3} + 2 Z_{2} Z_{1} Z_{3} + 4 Z_{2} Z_{3} Z_{1} - 4 Z_{3} Z_{2} Z_{1} 
    \\&\qquad + 5 Z_{1} Z_{2} Z_{1} Z_{2} + 7 Z_{1} Z_{2}^{2}Z_{1} - 5 Z_{2} Z_{1}^{2}Z_{2} - 7 Z_{2} Z_{1} Z_{2} Z_{1}
    \\&\qquad - 14 Z_{1} Z_{2} Z_{1}^{3} + 14 Z_{2} Z_{1}^{4}
\\
\phi_{5}(Z_{1}) &= -Z_{1} Z_{5} + Z_{5} Z_{1} + 3 Z_{1} Z_{2} Z_{3} + 4 Z_{1} Z_{3} Z_{2} + 5 Z_{1} Z_{4} Z_{1} - 3 Z_{2} Z_{1} Z_{3}
\\&\qquad - 4 Z_{3} Z_{1} Z_{2} - 5 Z_{4} Z_{1}^{2} - 9 Z_{1} Z_{2} Z_{1} Z_{2} - 12 Z_{1} Z_{2}^{2}Z_{1} - 14 Z_{1} Z_{3} Z_{1}^{2} 
\\&\qquad + 9 Z_{2} Z_{1}^{2}Z_{2} + 12 Z_{2} Z_{1} Z_{2} Z_{1} + 14 Z_{3} Z_{1}^{3} + 28 Z_{1} Z_{2} Z_{1}^{3} - 28 Z_{2} Z_{1}^{4}
\end{align*}
\caption{Some $\phi_k(Z_l)$ in $H_\ast M\xi$.}\label{figphikz}
\end{figure}

We next look at the $\Upsilon_{p,q}\in M\xi_\ast$, defined by 
$yx = \sum_{p,q} \Upsilon_{p,q} x^{1+p}y^{1+q}$.
We already know $\Upsilon_{0,0} = 1$ and $\Upsilon_{0,\ast} = \Upsilon_{\ast,0} = 0$.
Mapping this equation into $H_\ast M\xi[[T_1,T_2]]$ lets us compute the $\Upsilon_{p,q}$ explicitly
(see figure \ref{figupsZ}).

\begin{figure}
\begin{align*}
    \Upsilon_{1,2} =& - \Upsilon_{2,1} =  Z_{1} Z_{2} - Z_{2} Z_{1}
    \\
    \Upsilon_{1,3} =& - \Upsilon_{3,1} =  Z_{1} Z_{3} - Z_{3} Z_{1} - 3 Z_{1} Z_{2} Z_{1} + 3 Z_{2} Z_{1}^{2}
    \\
    \Upsilon_{2,2} =&  \hphantom{+} 0
    \\
    \Upsilon_{1,4} =& -\Upsilon_{4,1} = \ Z_{1} Z_{4} - Z_{4} Z_{1} - 3 Z_{1} Z_{2}^{2} - 4 Z_{1} Z_{3} Z_{1}  + 3 Z_{2} Z_{1} Z_{2} 
    \\ & \qquad + 4 Z_{3} Z_{1}^{2} + 9 Z_{1} Z_{2} Z_{1}^{2} - 9 Z_{2} Z_{1}^{3}
    \\
    \Upsilon_{2,3} =& -\Upsilon_{3,2} =  Z_{2} Z_{3} - Z_{3} Z_{2} + 2 Z_{1} Z_{2}^{2} - 2 Z_{1} Z_{3} Z_{1} - 2 Z_{2} Z_{1} Z_{2} 
    \\ & \qquad+ 2 Z_{3} Z_{1}^{2} + Z_{1} Z_{2} Z_{1}^{2} - Z_{2} Z_{1}^{3}
    \\
    \Upsilon_{1,5} =& -\Upsilon_{5,1} = Z_{1} Z_{5} - Z_{5} Z_{1} - 3 Z_{1} Z_{2} Z_{3} - 4 Z_{1} Z_{3} Z_{2} - 5 Z_{1} Z_{4} Z_{1}  + 3 Z_{2} Z_{1} Z_{3}
    \\ & \qquad + 4 Z_{3} Z_{1} Z_{2} + 5 Z_{4} Z_{1}^{2} + 9 Z_{1} Z_{2} Z_{1} Z_{2} + 12 Z_{1} Z_{2}^{2}Z_{1} + 14 Z_{1} Z_{3} Z_{1}^{2} 
    \\ & \qquad - 9 Z_{2} Z_{1}^{2}Z_{2} - 12 Z_{2} Z_{1} Z_{2} Z_{1} - 14 Z_{3} Z_{1}^{3} - 28 Z_{1} Z_{2} Z_{1}^{3} + 28 Z_{2} Z_{1}^{4}
    \\
    \Upsilon_{2,4} =& \hphantom{+} Z_{2} Z_{4} - Z_{4} Z_{2} + 2 Z_{1} Z_{2} Z_{3} - 2 Z_{1} Z_{4} Z_{1} - 2 Z_{2} Z_{1} Z_{3} - 4 Z_{2} Z_{3} Z_{1} + 4 Z_{3} Z_{2} Z_{1} 
    \\ & \qquad + 2 Z_{4} Z_{1}^{2} + Z_{1} Z_{2} Z_{1} Z_{2} - 7 Z_{1} Z_{2}^{2}Z_{1} + 8 Z_{1} Z_{3} Z_{1}^{2} - Z_{2} Z_{1}^{2}Z_{2} + 7 Z_{2} Z_{1} Z_{2} Z_{1} 
    \\ & \qquad - 8 Z_{3} Z_{1}^{3} - 4 Z_{1} Z_{2} Z_{1}^{3} + 4 Z_{2} Z_{1}^{4} 
    \\
    \Upsilon_{3,3} =&  -6 Z_{1} Z_{2} Z_{1} Z_{2} + 6 Z_{1} Z_{2}^{2}Z_{1} + 6 Z_{2} Z_{1}^{2}Z_{2} - 6 Z_{2} Z_{1} Z_{2} Z_{1}
    \\
    \Upsilon_{4,2} =& -\Upsilon_{2,4} +  6 Z_{1} Z_{2} Z_{1} Z_{2} - 6 Z_{1} Z_{2}^{2}Z_{1} - 6 Z_{2} Z_{1}^{2}Z_{2} + 6 Z_{2} Z_{1} Z_{2} Z_{1}
\end{align*}
\caption{Some $\Upsilon_{i,j}$ for the triangular structure of $M\xi$.}\label{figupsZ}
\end{figure}

Let $\Tria \subset M\xi_\ast$ be the $\BFK$-Hopf module subalgebra 
generated by the $\Upsilon_{p,q}$.

We consider the decreasing filtration of $H_\ast M\xi$ by polynomial degree in the $Z_k$,
i.e.~we set $\vert Z_{i_1}^{a_1}\cdots Z_{i_n}^{a_n} \vert = a_1+\cdots+a_m$ 
and let $F_k$ be the span of monomials $m$ with $\vert m\vert \ge k$.
The following lemma shows that the $\BFK$-action is nicely compatible with the 
filtration $F_k$:
\begin{lemma}\label{lem:liefilt}\quad
    \begin{enumerate}
        \item 
            For $a\in F_k$ and $l>0$ one has $\phi_l(a) \equiv [Z_l,a] \bmod F_{k+2}$.
        \item 
            One has $\Upsilon_{p,q} \equiv [Z_p,Z_q] \bmod F_{3}$.
    \end{enumerate}
\end{lemma}
\begin{proof}
    We use (\ref{phikZ}) for the claims about $\phi_j(a)$.
    One easily computes $\phi_0(a) = a$ and $\phi_1(a) = [Z_1,a]$ for all $a$.
    Now assume inductively that $\phi_j(a) \equiv [Z_j,a] \bmod F_{k+2}$ for $j=1,\ldots,l-1$.
    For the right-hand side of (\ref{phikZ}) one finds 
    \[ \sum_{j=0}^{l-1} \phi_j(a) T^{1+j}  + \phi_l(a) T^{1+l} \bmod (F_{k+2}, T^{2+l}). \]
    Comparing coefficients of $T^{1+l}$ gives the desired $\phi_l(a) \equiv [Z_l,a] \bmod F_{k+2}$.
    
    To compute $\Upsilon_{p,q}$ we use the observation of Remark \ref{rem:ypqviaphi} 
    that the braiding $\Upsilon$ can itself be viewed as a special case of the $\BFK$-action:
    we write $yx = \sum \phi_k(x)y^{1+k}$ and find formally 
    \[\phi_k(x) = \sum \Upsilon_{p,k} x^{1+p}.\]
    From $x = T + Z_1T^2 + Z_2T^3 + \cdots$ we deduce 
    \[ \phi_k(x) = \sum_{j\ge  0} \phi_k(Z_j) T^{1+j} 
         \equiv \sum_{j\ge 0}\, [Z_k,Z_j] T^{1+j} \bmod F_3. \]
    Since we also have $T = x - Z_1x^2 + (Z_2-2Z_1^2) x^3 - \cdots \equiv \sum (-1)^j Z_jx^{1+j}$ 
    we find $\phi_k(x) \equiv \sum_{j\ge 0} [Z_k,Z_j] x^{1+j} \bmod F_3$, as claimed.
\end{proof}

We can now complete the proof of theorem \ref{thm:triangular} and show that $\Tria\subset M\xi_\ast$ coincides with the universal triangular structure
$\Tria$ that we considered before:
\begin{proof}[Proof of Thm.~\ref{thm:triangular}, last part]
Recall from remark \ref{rem:triauniv} that as an algebra $\Tria$ is generated by the 
$\phi_{i_1}\cdots \phi_{i_n} \Upsilon_{p,q}$ with $i_1\ge i_2\ge \cdots \ge i_n \ge p < q$.  We here need to show that $\Tria$ is actually the \emph{free} associative algebra on these generators.

We use the filtration $F_k$ on $\Tria$ that we constructed above.
It shows that the associated graded $\overline\Tria$ of $\Tria$ is embedded in $H_\ast M\xi = \ZZ\langle Z_1,Z_2,\ldots\rangle$ 
as the subalgebra generated by the iterated commutators 
\[C_{i_1,i_2,\ldots,i_n,p,q} = \big[Z_{i_1},[Z_{i_2},[\cdots,[Z_{i_n},[Z_p,Z_q]]]]\cdots\big].\]
In other words, if $L$ denotes the free Lie algebra generated by the $Z_k$ 
then $\overline \Tria$ is the associative enveloping algebra of the Lie algebra $[L,L]$.
It is classical that $[L,L]$ is itself a free Lie algebra and a family of free generators is given by the $C_{i_1,i_2,\ldots,i_n,p,q}$
(see, for example, \cite{MR4368864}*{Corollary 2.16 (ii)}). The claim follows. 
\end{proof}

\end{subsection}

\begin{subsection}{The formal group law and generators for $M\xi_\ast$}

It remains to compute some coefficients of the formal group law of 
$M\xi$. The Hurewicz map takes 
\begin{equation}\label{aijdef}
    \Delta(x) = \sum_{i,j} a_{i,j}x^iy^j
\end{equation}
to
\begin{equation*}
    \sum_{k\ge 0} Z_k(T_1+T_2)^{1+k}
    = \sum_{i,j} a_{i,j} \, 
    \left( \sum_{p\ge 0} Z_p T_1^{1+p}\right)^i \cdot 
    \left( \sum_{q\ge 0} Z_q T_2^{1+q}\right)^j.
\end{equation*}
This can be used to compute the $a_{i,j}$ explicitly (see figure \ref{figaij}).

\begin{figure}[!htbp]
\begin{align*} 
a_{1,1} &= 2 Z_{1}
&
a_{1,2} &= 3 Z_{2} - 2 Z_{1}^{2}
\\
a_{2,1} &= 3 Z_{2} - 2 Z_{1}^{2}
&
a_{1,3} &= 4 Z_{3} - 2 Z_{1} Z_{2} - 6 Z_{2} Z_{1} + 4 Z_{1}^{3}
\\
a_{2,2} &= 6 Z_{3} - 6 Z_{2} Z_{1} + 2 Z_{1}^{3}
&
a_{3,1} &= 4 Z_{3} - 2 Z_{1} Z_{2} - 6 Z_{2} Z_{1} + 4 Z_{1}^{3}
\end{align*}
\begin{align*}
a_{1,4} &= 5 Z_{4} - 2 Z_{1} Z_{3} - 6 Z_{2}^{2} - 12 Z_{3} Z_{1} + 4 Z_{1}^{2}Z_{2} + 6 Z_{1} Z_{2} Z_{1} + 15 Z_{2} Z_{1}^{2} - 10 Z_{1}^{4}
\\
a_{2,3} &= 10 Z_{4} - 3 Z_{2}^{2} - 16 Z_{3} Z_{1} + 2 Z_{1} Z_{2} Z_{1} + 12 Z_{2} Z_{1}^{2} - 4 Z_{1}^{4}
\\
a_{3,2} &= 10 Z_{4} - 3 Z_{2}^{2} - 16 Z_{3} Z_{1} + 2 Z_{1}^{2}Z_{2} + 12 Z_{2} Z_{1}^{2} - 4 Z_{1}^{4}
\\
a_{4,1} &= 5 Z_{4} - 2 Z_{1} Z_{3} - 6 Z_{2}^{2} - 12 Z_{3} Z_{1} + 4 Z_{1}^{2}Z_{2} + 6 Z_{1} Z_{2} Z_{1} + 15 Z_{2} Z_{1}^{2} - 10 Z_{1}^{4}
\\
a_{1,5} &= 6 Z_{5} - 2 Z_{1} Z_{4} - 6 Z_{2} Z_{3} - 12 Z_{3} Z_{2} - 20 Z_{4} Z_{1} + 4 Z_{1}^{2}Z_{3} + 6 Z_{1} Z_{2}^{2} + 8 Z_{1} Z_{3} Z_{1}
\\&\qquad  + 15 Z_{2} Z_{1} Z_{2} + 21 Z_{2}^{2}Z_{1} + 36 Z_{3} Z_{1}^{2} - 10 Z_{1}^{3}Z_{2} - 14 Z_{1}^{2}Z_{2} Z_{1}
\\&\qquad  - 18 Z_{1} Z_{2} Z_{1}^{2} - 42 Z_{2} Z_{1}^{3} + 28 Z_{1}^{5}
\\
a_{2,4} &= 15 Z_{5} - 3 Z_{2} Z_{3} - 12 Z_{3} Z_{2} - 35 Z_{4} Z_{1} + 2 Z_{1} Z_{3} Z_{1} + 6 Z_{2} Z_{1} Z_{2} + 15 Z_{2}^{2}Z_{1} 
\\&\qquad + 42 Z_{3} Z_{1}^{2} - 4 Z_{1}^{2}Z_{2} Z_{1} - 6 Z_{1} Z_{2} Z_{1}^{2} - 30 Z_{2} Z_{1}^{3} + 10 Z_{1}^{5}
\\
a_{3,3} &= 20 Z_{5} - 8 Z_{3} Z_{2} - 40 Z_{4} Z_{1} + 2 Z_{1} Z_{2}^{2} + 6 Z_{2} Z_{1} Z_{2} + 6 Z_{2}^{2}Z_{1} + 40 Z_{3} Z_{1}^{2}
\\&\qquad - 4 Z_{1}^{3}Z_{2} - 4 Z_{1} Z_{2} Z_{1}^{2} - 24 Z_{2} Z_{1}^{3} + 8 Z_{1}^{5}
\\
a_{4,2} &= 15 Z_{5} - 3 Z_{2} Z_{3} - 12 Z_{3} Z_{2} - 35 Z_{4} Z_{1} + 2 Z_{1}^{2}Z_{3} + 12 Z_{2} Z_{1} Z_{2} + 9 Z_{2}^{2}Z_{1} 
\\&\qquad + 42 Z_{3} Z_{1}^{2} - 4 Z_{1}^{3}Z_{2} - 6 Z_{1}^{2}Z_{2} Z_{1} - 30 Z_{2} Z_{1}^{3} + 10 Z_{1}^{5}
\\
a_{5,1} &= 6 Z_{5} - 2 Z_{1} Z_{4} - 6 Z_{2} Z_{3} - 12 Z_{3} Z_{2} - 20 Z_{4} Z_{1} + 4 Z_{1}^{2}Z_{3} + 6 Z_{1} Z_{2}^{2} + 8 Z_{1} Z_{3} Z_{1}
\\&\qquad  + 15 Z_{2} Z_{1} Z_{2} + 21 Z_{2}^{2}Z_{1} + 36 Z_{3} Z_{1}^{2} - 10 Z_{1}^{3}Z_{2} - 14 Z_{1}^{2}Z_{2} Z_{1}
\\&\qquad  - 18 Z_{1} Z_{2} Z_{1}^{2} - 42 Z_{2} Z_{1}^{3} + 28 Z_{1}^{5}
\end{align*}
\caption{Some coefficients $a_{i,j}$ of the formal group law of $M\xi$.}\label{figaij}
\end{figure}

Recall that in the case of $MU$ the coefficients of its formal group law can be used to define 
polynomial generators of the Lazard ring $MU_\ast$.
The following Lemma is well-known:
\begin{lemma}\label{mugens}
    Let $F(x,y) = \sum b_{i,j}x^iy^j$ be the formal group law of $MU$.
    For every $n> 1$ choose integers $\lambda^{(n)}_1, \lambda^{(n)}_2, \ldots,
    \lambda^{(n)}_{n-1}$ such that 
    \[\gcd\left( \binom{n}{1}, \binom{n}{2}, \ldots, \binom{n}{n-1} \right) 
    = \sum_{k=1}^{n-1} \lambda^{(n)}_k \binom{n}{k}.\]
    Then $x_n = \sum_{k=1}^{n-1} \lambda^{(n+1)}_k b_{k,n-k+1}$ 
    is a family of free polynomial generators for $MU_\ast$.
\end{lemma}
\begin{proof}
    See \cite{MR3438568}*{Remark 4.4}.
\end{proof}

We will show in section \ref{sec:fguniv} that the same procedure also gives generators of $M\xi$:
\begin{thm}\label{thm:mxigens}
    Let $a_{i,j}$ be the coefficients of the formal group law of $M\xi$ as in (\ref{aijdef})
    and choose integers $\lambda^{(n)}_k$ as in Lemma \ref{mugens} and define 
    $X_n\in M\xi_n$ as
    \[X_n = \sum_{k=1}^{n-1} \lambda^{(n)}_k a_{k,n-k+1}.\]
    Then $M\xi_\ast$ is the free $\Tria_\ast$-module with basis 
    given by the monomials $X_1^{a_1}X_2^{a_2}\cdots X_n^{a_n}$. 
\end{thm}

We emphasize once more that the $X_k$ here are non-central. 
Here are some non-trivial commutators that illustrate this:
we are using $X_1 = a_{1,1}$, $X_2 = a_{1,2}$, $X_3 = -a_{1,3} + a_{2,2}$.
\begin{align*} 
    [X_1,X_2] &= 6 \Upsilon_{1,2} 
    \\
    [X_1,X_3] &= 
    4 \phi_1(\Upsilon_{1,2}) + 4 \Upsilon_{1,3} + 8 \Upsilon_{1,2} X_1
    \\ 
    [X_2,X_3] &=  
    2 \phi_{1,1}(\Upsilon_{1,2})
    -4 \phi_{1}(\Upsilon_{1,3})
    + 6 \Upsilon_{2,3} 
    -3 \phi_1(\Upsilon_{1,2}) X_1 
    \\&+2 \Upsilon_{1,3}X_2 
    -6 \Upsilon_{1,2}X_1
    + \Upsilon_{1,2}X_1^2
\end{align*} 

\end{subsection}

\begin{subsection}{Chern classes in $M\xi$}\label{sec:chern}
We now compute explicitly some Chern classes in $M\xi^\ast BU$
expressed as symmetric functions in $M\xi^\ast[[x_1,x_2,\ldots]]$.
We use the method developed in \cite{MR1398918}.
As explained in section \ref{sec:symm} this first requires 
the computation of the Vieta coordinates $y_k = v_k x_k v_k^{-1}$ where 
\[ 
    v_k = \left| \begin{matrix}x_1^{k-1} & x_2^{k-1} & \cdots & x_k^{k-1} \\ \vdots & & & \vdots \\ 1 & 1 & \cdots & 1 \end{matrix}\right|_{1k}
\] is a certain Vandermonde quasi-determinant. The Chern class $c_k\in M\xi^\ast BU(k)$ is then given by the product $c_k = y_k y_{k-1} \cdots y_1$.

The $y_k$ are power series in $x_1,\ldots,x_k$ that can only be computed approximately (except for $y_1=1$, of course).
We here give some values up to $x$-degree $7$:
\begin{align*}
    y_2 &\equiv x_2 -\Upsilon_{1,2} x_1^{2}x_2^{2} -\Upsilon_{1,3} x_1^{3}x_2^{2} -\Upsilon_{1,3} x_1^{2}x_2^{3} -\Upsilon_{1,4} x_1^{4}x_2^{2}
    \\& - \left(\Upsilon_{1,4} + \Upsilon_{2,3}\right) x_1^{3}x_2^{3} 
     -\Upsilon_{1,4} x_1^{2}x_2^{4} -\Upsilon_{1,5} x_1^{5}x_2^{2} 
     \\& + \left(2 \Upsilon_{1,2}^2 - \Upsilon_{1,5} - \Upsilon_{2,4}\right) x_1^{4}x_2^{3} 
    + \left(6 \Upsilon_{1,2}^2 - \Upsilon_{1,5} - \Upsilon_{2,4}\right) x_1^{3}x_2^{4} -\Upsilon_{1,5} x_1^{2}x_2^{5} 
    \\
    y_3 &\equiv 
    x_3 -\Upsilon_{1,2} x_1^{2}x_3^{2} -\Upsilon_{1,2} x_2^{2}x_3^{2} 
    -\Upsilon_{1,3} \left( x_1^{3}x_3^{2} + x_1^{2}x_3^{3} + x_2^{3}x_3^{2} + x_2^{2}x_3^{3} \right) 
    \\& + \left(-\Upsilon_{1,4} - \Upsilon_{2,3}\right) x_1^{3}x_3^{3} + \left(\phi_{1}(\Upsilon_{1,3}) - \phi_{2}(\Upsilon_{1,2})\right) x_1^{2}x_2^{2}x_3^{2} 
    \\& -\Upsilon_{1,4} \left( x_1^{4}x_3^{2} + x_1^{2}x_3^{4} + x_2^{4}x_3^{2} + x_2^{2}x_3^{4} \right) 
    + \left(-\Upsilon_{1,4} - \Upsilon_{2,3}\right) x_2^{3}x_3^{3} + 4  \Upsilon_{1,2}^2 x_2^{3}x_3^{4} 
    \\& -\Upsilon_{1,5} x_1^{5}x_3^{2} + \left(2 \Upsilon_{1,2}^2 - \Upsilon_{1,5} - \Upsilon_{2,4}\right) x_1^{4}x_3^{3} + \left(\Upsilon_{1,2}^2 + \phi_{1}(\Upsilon_{1,4}) - \phi_{3}(\Upsilon_{1,2})\right) x_1^{3}x_2^{2}x_3^{2} 
    \\& + \left(6 \Upsilon_{1,2}^2 - \Upsilon_{1,5} - \Upsilon_{2,4}\right) x_1^{3}x_3^{4} 
    + \left(\Upsilon_{1,2}^2 + \phi_{1}(\Upsilon_{1,4}) - \phi_{3}(\Upsilon_{1,2})\right) \left( x_1^{2}x_2^{3}x_3^{2} + x_1^{2}x_2^{2}x_3^{3} \right) 
    \\& 
    -\Upsilon_{1,5} \left( x_1^{2}x_3^{5} + x_2^{5}x_3^{2} + x_2^{2}x_3^{5} \right)
    + \left(2 \Upsilon_{1,2}^2 - \Upsilon_{1,5} - \Upsilon_{2,4}\right) \cdot \left( x_2^{4}x_3^{3} + x_2^{3}x_3^{4} \right) 
\end{align*}
This gives the following for $c_2$ and $c_3$:
we use the quasi-symmetric monomials $m_I = \sum_{j_1<\cdots<j_k} x_{j_1}^{i_1}\cdots x_{j_k}^{i_k}$ 
to express these functions more succinctly. The formulas are complete modulo terms of $x$-degree larger than $8$.
\begin{align*}
    c_2 &\equiv 
    m_{11} -\Upsilon_{1,2} m_{23} -\Upsilon_{1,3} \left( m_{33} + m_{24} \right)
    -\Upsilon_{1,4} \left( m_{43} + m_{25} + m_{34} \right) 
    -\Upsilon_{2,3} m_{34} 
    \\& -\Upsilon_{1,5} \left( m_{53} + m_{26} \right) 
    + \left(2 \Upsilon_{1,2}^2 - \Upsilon_{1,5} - \Upsilon_{2,4}\right) \left( m_{44} + m_{35} \right)
    + 4  \Upsilon_{1,2}^2 m_{35}
    \\
    c_3 &\equiv m_{111} 
    -\Upsilon_{1,2} \left( m_{231} + m_{213} + m_{123} \right) 
    -\phi_{1}(\Upsilon_{1,2}) m_{223} 
    \\& 
    -\Upsilon_{1,3} \left( m_{331} + m_{313} + m_{241} + m_{214} + m_{133} + m_{124}\right) 
    \\& -\Upsilon_{1,4} \left( m_{431} 
    + m_{413} 
    + m_{251} 
    + m_{215} 
    + m_{143} 
    + m_{125}
    + m_{341}
    + m_{314}
    + m_{134}
    \right)
    \\& - \Upsilon_{2,3} \left( m_{341} + m_{314} + m_{134} \right)
    -\phi_{2}(\Upsilon_{1,2}) \left( m_{323} + m_{233} \right) 
    -\phi_{1}(\Upsilon_{1,3}) m_{224} 
\end{align*}
\end{subsection}

\end{section}
\begin{section}{The formal group of $M\xi$ is universal}\label{sec:fguniv}
We now want to prove that the formal group of $M\xi$ has the expected universal property:
\begin{thm}\label{thm:mxiuniv}
Let $(R,\Upsilon)$ be a strict triangular $\BFK$ Hopf module algebra and $F=(\Upsilon,\Delta,\chi)$ a 
one-dimensional, commutative formal group law over $R$.
Then there exists a unique map $\phi:M\xi_\ast \rightarrow R$
of triangular $\BFK$ Hopf module algebras such that $\phi_\ast F_{M\xi} = F$.  
\end{thm}

In the classical non-braided case this theorem is due to Lazard \cite{MR77542} and the 
ring $MU_\ast$ that carries the universal formal group law is known as the 
Lazard ring.  While the statement of Lazard's theorem is very elegant
all of its proofs are essentially computational; so far every attempt to find 
a natural, non-computational proof has ended in a mirage \cite{MR2683228}*{section 5.3.2}.
Our proof of the braided generalization in theorem \ref{thm:mxiuniv} is no exception.
    
We will follow the inductive approach to classify formal group law chunks by induction on their length.
Recall that two formal group laws $F$ and $G$ represent the same formal group law chunk of length $n$
if $x+_Fy$ and $x+_Gy$ agree modulo terms of polynomial degree $n+1$. 

\newcommand{\univ}{{\mathrm{univ}}}
\begin{lemma}\label{lem:fgchunks}
    Let $M_n = \Tria_\univ [X_1,\ldots,X_n]$ be the subalgebra of $M\xi_\ast$ generated by the 
    triangular subalgebra $\Tria_\univ$ and the first $n$ generators of theorem \ref{thm:mxigens}.
    Then $M_n$ represents the functor that classifies formal group law chunks of length $\le n$.
\end{lemma}
\begin{proof}
    Since we consider the underlying triangular structure to be fixed we are already given a unique 
    $\phi:\Tria_\univ \rightarrow R$ that is independent of the formal group law data $(\Delta,\chi)$.
    We need to show that this $\phi$ can be extended on the additional generators $X_1,\ldots,X_n$.

    We do this by induction on $n$, the case $n=0$ being trivial. 
    Our proof is modeled on the presentation in \cite{MR860042}*{Appendix 2} and
    the references in this proof refer to that chapter. 
    Let $F(x,y) = x+_Fy$ be a given formal group law chunk of length $n$
    and $\phi_{n-1}:M_{n-1}\rightarrow R$ the map that we need to extend.
    As in [A.2.12] we let $\Gamma(x,y)$ be the degree $n$ component 
    of the difference between $G=\left(\phi_{n-1}\right)_\ast F_{M\xi}$ and $F$.
    Using the fact that commutators in $R[[x,y]]$ raise the polynomial degree by at least one
    we find that $\Gamma$ satisfies the usual cocycle conditions (compare [A2.1.28])
    \begin{align*} 
        &\Gamma(x,y) = \Gamma(y,x),\quad \Gamma(x,0) = \Gamma(0,x) = 0,\\
        &\Gamma(x,y) +\Gamma(x+y,z) = \Gamma(x+y,z) + \Gamma(y,z).
    \end{align*}
    By Lazard's comparison lemma [A.2.1.12] there is a unique $c\in R$ such that 
    $\Gamma(x,y) = c\cdot C_n(x,y)$ with 
    $C_n(x,y) = c_n^{-1}\left((x+y)^n - x^n - y^n\right)$
    where $c_n$ is the gcd from theorems \ref{mugens}, \ref{thm:mxigens}.
    It follows that we can adjoin a non-commuting variable $\hat X_n$ to $M_{n-1}$ to get the required extension 
    \[ \hat \phi_n : \Tria_\univ[X_1,\ldots,X_{n-1}] \langle \hat X_n \rangle \rightarrow R \]
    with $\hat X_n \mapsto c$.
    It remains to show that this descends through the projection 
    \[ \Tria_\univ[X_1,\ldots,X_{n-1}] \langle \hat X_n \rangle \rightarrow \Tria_\univ[X_1,\ldots,X_n]. \]
    This requires that the commutation relations between $X_n$ and an element $m\in M_{n-1}$ are the same as 
    those between $c$ and $\phi_{n-1}(m)$.  But this follows from Schauenburg's formula (\ref{schauenburg}) because 
    it expresses the commutator $\Upsilon(c\otimes m)$ through the formal group law data $\Delta$ and $\chi$.
\end{proof}

This also completes the proofs of theorems \ref{thm:mxiuniv} and \ref{thm:mxigens}.

\end{section}
\begin{section}{A splitting of $M\xi$}\label{sec:splitting}
    We now want to use the section $\sigma = \sigma^{MU} : MU \rightarrow M\xi$ to show that $M\xi$ 
    splits (non-multiplicatively) into a sum of suspended copies of $MU$.
    Let $\TriSpace$ denote the homotopy fiber of $\Omega\Sigma \CP^\infty \rightarrow BU$.
    We will show below that this represents the universal triangular structure which explains the name.
    We have a map 
    \[\begin{tikzcd}
        \TriSpace_+ \land MU \ar[r,"\mathrm{incl}"] &  \Omega\Sigma \CC P^\infty_+ \land MU \ar[r, "\cong"] & 
        M\xi\land MU \ar[r, "\id\land\sigma"] & M\xi\land M\xi \ar[r, "\mathrm{mult}"] & M\xi
    \end{tikzcd}\]
    \begin{thm}\label{thm:splitting}
        This map is an equivalence. 
    \end{thm}
    To prove this we will show that the map induces an equivalence in integral cohomology.
    \begin{lemma}
        One has $H^\ast \TriSpace = \QSym \otimes_{\Sym} \ZZ = \QSym /\!/ \Sym$.
    \end{lemma}
    \begin{proof}
        One can use the Eilenberg-Moore spectral sequence
        to compute the cohomology of $\TriSpace$ as the homotopy pullback of  
        \[ \begin{tikzcd} & \Omega\Sigma\CP^\infty \ar[d] \\ \{\mathrm{pt}\}  \ar[r] & BU \end{tikzcd} \]
        Its $E_2$ page is $\Tor_{H^\ast BU}\left(H^\ast \Omega\Sigma\CP^\infty, \ZZ\right)$.
        Since $H^\ast \Omega\Sigma\CP^\infty = \QSym$ is a free module over the subalgebra $H^\ast BU = \Sym$
        this degenerates to the tensor product $\QSym \otimes_{\Sym} \ZZ$ and the spectral sequence collapses.
    \end{proof}

    It follows that $H_\ast \TriSpace$ can be described dually as the cotensor product 
    \[ \NSym \cotensor_{\Sym} \ZZ = \left\{ n\in \NSym \,:\, (\id\otimes\pi) \Delta n = n\otimes 1 \right\} \]
    where $\NSym = H_\ast \Omega\Sigma\CP^\infty = \ZZ\langle Z_1,Z_2,\ldots\rangle$ 
    is the ring of non-commutative symmetric functions and $\pi : \NSym \rightarrow \Sym$ is the 
    usual projection. 
    Let $\widetilde{\Sym} \subset \NSym$ be spanned by the monomials $Z_{i_1}^{a_1}\cdots Z_{i_n}^{a_n}$ with 
    $i_1\le \cdots \le i_n$.
    Clearly the composite $\widetilde{\Sym} \rightarrow \NSym \rightarrow \Sym$ is a bijection.
    \begin{lemma}\label{lem:reorder}
        The multiplication map $\widetilde{\Sym} \otimes \left(\NSym \cotensor_{\Sym} \ZZ \right) \rightarrow \NSym$
        is an isomorphism.
    \end{lemma}
    \begin{proof}
        We use the filtration of $\NSym=\ZZ\langle Z_1,Z_2,\ldots\rangle$ by powers of the augmentation ideal $I$.
        Let $V=\ZZ\{Z_1,Z_2,\ldots\}$ be the free module on the generators $Z_k$. One can think of $\NSym$ as the universal 
        enveloping algebra $\UAss(L)$ of the free Lie algebra $L\subset \NSym$ generated by $V$. 
        Since $L= V \oplus [L,L]$ one has (using the Poincaré-Birkhoff-Witt theorem twice)
        \[ 
            \NSym = \UAss(L) \cong \Sym(V \oplus [L,L]) \cong \widetilde{\Sym} \cdot \Sym([L,L]) = \widetilde{\Sym} \cdot \UAss([L,L])
        \] 
        The claim will follow if we show that the associated graded of $\NSym \cotensor_{\Sym} \ZZ$ with respect to the $I$-adic filtration 
        can be identified with $\UAss([L,L])$.
        For this we need to show that 
        \begin{enumerate}
            \item there are $A_{i,j} \in \NSym \cotensor_{\Sym} \ZZ$ for all $i,j$ with 
        $A_{i,j} = [Z_i,Z_j]$ mod $I^3$ and that 
            \item to every $a\in \left(\NSym \cotensor_{\Sym} \ZZ\right) \cap I^k$ 
        there are $\Theta_j(a)\in \NSym \cotensor_{\Sym} \ZZ$ with $\Theta_j(a) - [Z_j,a] \in I^{k+2}$. 
        \end{enumerate}
        As in Lemma \ref{lem:hurlem}
        assume variables $x_1,x_2$ with a representation through central variables $T_k$
        \[ x_k = T_k + Z_1 T_k^2 + Z_2 T_k^3 + \cdots. \] 
        The coproduct $\Delta Z_n = \sum_{p+q=n} Z_p \otimes Z_q$ has $\Delta x_j = T_j^{-1} x_j \otimes x_j$.
        Write $x_2x_1 = B\cdot x_1x_2$ with $B=\sum_{i,j} B_{i,j}T_1^iT_2^j$.
        Then 
        \begin{align*} 
            (\id\otimes\pi)\Delta(x_2x_1) 
            &= (\id\otimes\pi)\Delta(B x_1x_2) \\ 
            & = (\id\otimes\pi)\left( \Delta B \right) \cdot (\id\otimes\pi)\Delta(x_1x_2) \\
            &= T_1^{-1} T_2^{-1} (\id\otimes\pi)\left( \Delta B \right) \cdot (x_1x_2 \otimes \pi(x_1x_2) ) \\
        \intertext{and also}
            (\id\otimes\pi)\Delta(x_2x_1) 
            &= T_1^{-1} T_2^{-1} x_2x_1 \otimes \pi(x_2x_1) \\ 
            &= T_1^{-1} T_2^{-1} (Bx_1x_2) \otimes \pi(x_1x_2) \\
            &= T_1^{-1} T_2^{-1} (B\otimes 1) \cdot \left(x_1x_2\otimes \pi(x_1x_2)\right).
        \end{align*} 
        Cancelling the common factors $T_1^{-1}T_2^{-1} \left(x_1x_2\otimes \pi(x_1x_2)\right)$ 
        gives $(\id\otimes\pi)\Delta B = (B\otimes 1)$, i.e.~$B_{i,j}\in \NSym \cotensor_{\Sym} \ZZ$
        for all $i,j$.        
        Now $x_k^{-1} = \sum_{j\ge 0} \chi(Z_j) T^{j-1}$ where $\chi(Z_j) \equiv -Z_j$ mod $I^2$ is the 
        antipode corresponding to the coproduct $\Delta$. 
        From $B=x_2x_1x_2^{-1}x_1^{-1}$ one then finds that modulo $I^2$ the coefficient $B_{n,m}$ is given by 
        \begin{align*} 
            \sum_{\substack{i+k=n\\j+l=m}} Z_iZ_j\chi(Z_k)\chi(Z_l) 
            &\equiv Z_nZ_m + Z_n\chi(Z_m) + Z_m\chi(Z_n) + \chi(Z_n)\chi(Z_m) 
            \\ &\equiv [Z_n,Z_m] \bmod I^3 
        \end{align*} 
        as required in (1).

        For (2) we assume $a\in I^k$ with $(\id\otimes\pi)\Delta a = a\otimes 1$ and find similarly 
        that $x_1ax_1^{-1} \in \NSym \cotensor_{\Sym} \ZZ$.
        One has 
        \[ x_1ax_1^{-1} = \sum_n \left( \sum_{i+j=n} Z_ia\chi(Z_j) \right) T^n \equiv 
        \sum_n \left( Z_n a - a Z_n  \right) T^n \mod I^{k+1}. \]
    \end{proof}

    \begin{proof}[Proof of Theorem \ref{thm:splitting}]
        The previous lemma showed that the composite
        \[ \begin{tikzcd}[column sep=1.7cm] \Sym \otimes \left( \NSym \cotensor_{\Sym} \ZZ \right) 
            \ar[r, "\mathrm{repr} \otimes \mathrm{incl}"] & \NSym \otimes \NSym \ar[r, "\mathrm{mult}"] & \NSym 
        \end{tikzcd} \] 
        is an isomorphism where $\mathrm{repr}$ is a section to the abelianization map $\NSym \rightarrow \Sym$. 
        Dualizing shows that the composite 
        \[ \begin{tikzcd}[column sep=1.7cm] \Sym \otimes \left(\QSym /\!/ \Sym\right) & \QSym \otimes \QSym \ar[l, swap, "\sigma^\ast \otimes \mathrm{proj}" ]
            & \QSym \ar[l, swap, "\mathrm{coproduct}"]
        \end{tikzcd} \] 
        is an isomorphism as well.
        But this is easily seen to coincide with the map from $H^\ast M\xi$
        to $H^\ast MU \otimes H^\ast \TriSpace$ under consideration.
    \end{proof}

\end{section}
\begin{section}{Operations and cooperations of $M\xi$}
We now focus our attention on the problem of understanding the algebras of operations $M\xi^\ast M\xi$ and 
of cooperations $M\xi_\ast \opposite{M\xi}$.  
Conventional wisdom has it that the latter should be understood as a kind of Hopf algebroid (see \cite{MR860042}*{A1.1.1}).
This is somewhat complicated by the fact that these objects are both non-commutative and non-cocommutative. 
They thus belong to a class of Hopf algebras that first properly came into focus 
in the work of the 1980s on quantum groups.

For an algebra $A$ let $\Mod_A$ denote the category of left $A$-modules. If $A$ is provided with a coproduct 
$A\rightarrow A\otimes A$ this turns $\Mod_A$ into a monoidal category 
because the tensor product $M\otimes N$ of two $A$-modules then carries a natural $A$ action via 
$a(m\otimes n) = \sum (a'm) \otimes (a''m)$.

In general there will be no natural way to identify $M\otimes N$ and $N\otimes M$ as $A$-modules, though,  
since the naive identification $m\otimes n \leftrightarrow n\otimes m$ is not $A$-linear.
Such an identification therefore, if it exists, represents additional structure on $A$.
Recall from \cite{MR869575}*{section 10} that a Hopf algebra $A$ is 
called \emph{quasi-triangular} if such a natural braiding 
$\begin{tikzcd}[column sep=small]\Upsilon\colon M\otimes N \ar[r,"\simeq"] & N\otimes M\end{tikzcd}$ is provided.  
Furthermore $A$ is called \emph{triangular} if that braiding is actually symmetric, i.e. if $\Upsilon^2 = \id$. 

If $A$ arises as the algebra of operations of a non-commutative ring spectrum $E$ 
we morally think of an $A$-module $M\in\Mod_A$ 
as arising as the $E$-cohomology $M=E^\ast(X)$ of a space $X$. 
We therefore naturally expect a symmetric braiding 
$E^\ast X \otimes E^\ast Y \cong E^\ast Y \otimes E^\ast X$
related to the topologically induced flip $E^\ast (X\times Y) \cong E^\ast (Y\times X)$. 
In other words we are canonically inclined to think of $E^\ast E$ as a triangular Hopf algebra.   
That this works for $M\xi$ is something we will establish in section \ref{sec:triangular} below.

Understanding $M\xi_\ast \opposite{M\xi}$ is less straightforward: for a ring spectrum $E$ there is no reason to think 
that the augmentation $\epsilon\colon E_\ast\opposite{E} \rightarrow E_\ast$ 
or the coproduct $\Delta \colon E_\ast\opposite{E} \rightarrow E_\ast\opposite{E} \otimes_{E_\ast} E_\ast\opposite{E}$
should be multiplicative maps. Indeed, with $\eta_L:E\cong E\land S^0 \rightarrow E\land E$ and 
$\eta_R\colon E \cong S^0\land E \rightarrow E\land E$ one has 
$\epsilon(\eta_L(x) y) = x \epsilon(y)$ and $\epsilon(\eta_R(x)y) = \epsilon(y) x$ which shows that 
parts of $\epsilon$ behave multiplicatively, others anti-multiplicatively. 

We will see below that for $E=M\xi$ the cooperations can be decomposed as 
$M\xi_\ast\opposite{M\xi} = M\xi_\ast \otimes \BFK = \BFK \otimes M\xi_\ast$ where $\BFK$ is the Brouder-Frabetti-Krattenthaler Hopf algebra as above
and the two decompositions differ about using $\eta_L$ or $\eta_R$ for the embedding of $M\xi_\ast$. 
We will show that this embeds $\BFK$ into $M\xi_\ast \opposite{M\xi}$ as a Hopf algebra, 
i.e.~in a way that is compatible with the coproducts. 
This allows to carry out computations with $M\xi_\ast\opposite{M\xi}$ even though a good theoretical framework 
for this object seems not to be available.
We also get dually an embedding of the dual Hopf algebra $\BFK^\ast$ into the operation algebra $M\xi^\ast M\xi$.

An intriguing consequence of the non-commutativity of $M\xi$ is that the cooperation algebra can be realized 
as a subalgebra of the operation algebra: the map 
\begin{equation}\label{eq:phi} 
     \Phi \colon M\xi_\ast \opposite{M\xi} \rightarrow M\xi^\ast M\xi 
\end{equation}  
with \enquote{$a\land b \mapsto (x \mapsto axb)$} induces this embedding.  
This shows that in $M\xi^\ast M\xi$ there are copies of both $\BFK$ and of its dual $\BFK^\ast$. 
We show that these can be combined to define a map of Hopf algebras 
\begin{equation}\label{eq:drinfeld}
    D(\BFK) = \BFK \otimes \BFK^\circ \rightarrow M\xi^\ast M\xi
\end{equation}
where $D(\BFK)$ is Drinfeld's quantum double of $\BFK$.
That something like this might be true had been conjectured in \cite{morava2020renormalization}.

\begin{subsection}{The embedding $\Phi\colon M\xi_\ast \opposite{M\xi} \rightarrow M\xi^\ast M\xi$}
    As described above let $\Phi(z)$ for $z\in M\xi_\ast \opposite{M\xi}$ be given as the composite 
    \[ \begin{tikzcd} M \cong S^0 \land M \ar[r, "z\land\id"] & M\land M\land M 
        \ar[r,"\id\land\mathrm{flip}"] & M\land M\land M \ar[r, "\mu"] & M.
    \end{tikzcd} \]
    Here $\mu$ denotes the multiplication on $M\xi$.
    We will write $\Phi_z$ for $\Phi(z)$ where this is convenient.

    The following lemma is one reason why we use the opposite multiplication in the second factor of 
    $M\xi_\ast\opposite{M\xi}$:
    \begin{lemma} $\Phi$ is multiplicative: $\Phi(zw) = \Phi(z) \Phi(w)$.
    \end{lemma}
    \begin{proof}
        With \enquote{$z=a\land b$} and \enquote{$w=c\land d$} one gets 
        \enquote{$zw = ac\land db$} and
        \[
            \Phi_{zw}(x) = (ac)x(db) = a(cxd)b = \Phi_z\left(cxd\right) = \Phi_z\left(\Phi_w\left(x\right)\right)
        \]
    \end{proof}

    Let $x\colon{}\Sigma^2 \CC P^\infty \rightarrow M\xi$ be the complex orientation.
    We have
    \[(M\xi\land M\xi)^\ast \CC P^\infty = M\xi_\ast\opposite{M\xi} [[x^L]] = M\xi_\ast\opposite{M\xi} [[x^R]]\]
    where $x^L = x\land 1$, $x^R = 1\land x$.
    As in \cite{MR655040} we define $b_i\in M\xi_\ast\opposite{M\xi}$ via
    $x^L = \sum_{j\ge 0} b_j (x^R)^{j+1}$.
    One has $M\xi_\ast \opposite{M\xi} = M\xi_\ast\langle b_1,b_2,\ldots\rangle$ and $b_0=1$.

    On $M\xi_\ast$ we now have operators $\Phi_z$ in addition to the $\phi_k$ from the natural $\BFK$-action.
    Let $\Phi_i = \Phi(b_i)$.
    The next lemma shows that we have in fact $\Phi_i = \phi_i$.
    \begin{lemma}
        For $a\in M\xi^\ast$ one has $xa = \sum_{i\ge 0}\Phi_i(a)x^{i+1}$.
    \end{lemma}
    \begin{proof}
        Decompose $b_j$ formally as $b_j = \sum b_j'\land b_j''$. Then
        \begin{align*}
            xa &= \mu (x\land a) = \mu (1\land a) (x\land 1) = \mu (1\land a) x^L
            \\ &= \sum_j \mu (1\land a) b_j (x^R)^{j+1}
            \\ &= \sum_j \mu (1\land a) (b_j'\land b_j'') (1\land x^{j+1})
            \\ &= \sum_j \mu\left( b_j' \land a b_j'' x^{j+1} \right)
            \\ &= \sum_j b_j'ab_j'' x^{j+1} = \sum_j \Phi_{b_j}(a) x^{j+1}.
        \end{align*}
    \end{proof}

    We will show that $\Phi$ is nicely compatible with the coproduct
    $M\xi_\ast \opposite{M\xi} \rightarrow M\xi_\ast \opposite{M\xi} \otimes M\xi_\ast \opposite{M\xi}$.
    Using the known composition law of the $\Phi_i\in\BFK$ then lets us compute the $\Delta(b_k)$.

Now recall that the algebra of operations also carries a coproduct 
\[ \delta \colon M\xi^\ast M\xi \rightarrow M\xi^\ast\left(M\xi\land M\xi\right) \cong M\xi^\ast M\xi \hat\otimes M\xi^\ast M\xi. \]
For $\psi\in M\xi^\ast M\xi$ its coproduct $\delta \psi = \sum \psi' \hat\otimes \psi''$ is characterized by 
the multiplication rule $\psi(xy) = \sum \psi'(x) \psi''(y)$.
\begin{lemma}\label{phicoproduct}
    Let $Z\in M\xi_\ast\opposite{M\xi}$ with $\Delta Z = \sum Z'\otimes Z''$.
    Then
    \begin{equation}
        \Phi_Z(xy) = \sum \Phi_{Z'}(x) \Phi_{Z''}(y).
    \end{equation}
\end{lemma}
In more suggestive terms the lemma says that we can set $(\Phi_Z)' = \Phi_{Z'}$ and $(\Phi_Z)'' = \Phi_{Z''}$ here, 
so the map $\Phi$ is compatible with coproducts.
\begin{proof}
    Write $M=M\xi$.
    The map $\Phi$ has an extension to several variables
    \[\Phi^{(k)} :
    \pi_\ast \big(\underbrace{M\land \cdots \land M}_{\text{$k$ factors}}\big) \rightarrow
    M^\ast\big(\underbrace{M\land \cdots \land M}_{\text{$(k-1)$ factors}}\big)\]
    where
    \[\Phi^{(k)}(a_1\land\cdots\land a_k)(z_1\land \cdots \land z_{k-1}) =
    a_1z_1a_2z_2a_3\cdots a_{k-1}z_{k-1}a_k\]
    The multiplicative properties of $\Phi$ are elucidated by the commutative diagram
    \[\begin{tikzcd}
        {M_\ast\opposite{M} \otimes M_\ast\opposite{M}} && {M^\ast M \otimes M^\ast M} \\
        {\pi_\ast \left(M\land M \land M\right)} && {M^\ast(M\land M)} \\
        {M_\ast\opposite{M}} && {M^\ast M}
        \arrow["\Phi", from=3-1, to=3-3]
        \arrow["{\Phi^{(3)}}", from=2-1, to=2-3]
        \arrow["\Phi\otimes\Phi", from=1-1, to=1-3]
        \arrow["\alpha"', from=1-1, to=2-1]
        \arrow["\beta", from=1-3, to=2-3]
        \arrow["{\id\land 1 \land\id}", from=3-1, to=2-1]
        \arrow["{\mu^\ast}"', from=3-3, to=2-3]
    \end{tikzcd}\]
    Here \enquote{$\alpha((a\land b) \otimes (c\land d)) = a\land bc\land d$}.
    The map $\alpha$ is an isomorphism whereas $\beta$ becomes an isomorphism after completion of the source.
    The claim now follows since the  left vertical composite defines the diagonal
    \[\Delta \colon{} M_\ast\opposite{M} \rightarrow M_\ast\opposite{M} \otimes_{M_\ast} M_\ast\opposite{M}.\]
\end{proof}

\begin{corollary}\label{mxicoproduct}
    The coproduct 
    \[\Delta \colon{} M_\ast\opposite{M} \rightarrow M_\ast\opposite{M} \otimes_{M_\ast} M_\ast\opposite{M}\]
    restricts to the Brouder-Frabetti-Krattenthaler 
    diagonal $\Delta:\BFK\rightarrow \BFK\otimes\BFK$ on $\BFK \subset M\xi_\ast\otimes \BFK = M\xi_\ast\opposite{M\xi}$.
\end{corollary}
\begin{proof}
    The lemma shows that for $z\in\BFK$ the coproduct $\Delta z$, computed in $M\xi_\ast\opposite{M\xi}$,
    gives the product rule for $\Phi_z(ab) = \sum \Phi_{z'}(a)\Phi_{z''}(b)$.
    But we already know such a product rule from the comultiplication in $\BFK$. 
    Since there cannot be two different product rules here, the coproducts have to coincide.
\end{proof}

As a side result we see that the coproduct $\Delta$ on $M\xi_\ast\opposite{M\xi}$
is a multiplicative map when restricted to either $M\xi_\ast$ or $\BFK$. 
Of course, this does not mean that it would be multiplicative on all of $M\xi_\ast\opposite{M\xi} = M\xi_\ast \cdot \BFK$
because there is a non-trivial commutation rule between $M\xi_\ast$ and $\BFK$.
\end{subsection}

\begin{subsection}{The triangular structure on $M\xi^\ast M\xi$}\label{sec:triangular}
    We now spell out in which sense the operation algebra $M\xi^\ast M\xi$ can be thought of as a 
    triangular Hopf algebra.  As indicated earlier this is a rather natural assumption for  
    the operation algebra $E^\ast E$ of any ring spectrum $E$, assuming a suitable Künneth isomorphism 
    $E^\ast \left(E\land E\right) \cong E^\ast E \hat\otimes_{E^\ast} E^\ast E$. 
    This isomorphism allows to define the diagonal 
    \[ \begin{tikzcd} E^\ast E \ar[r, "\mu^\ast"] 
        \ar[rr, "\delta", bend right=30] & E^\ast(E\land E) \ar[r, "\cong"] & E^\ast E \hat\otimes_{E^\ast} E^\ast E.
    \end{tikzcd} \]
    It also allows to express the opposite multiplication 
    $\mu T \in E^\ast(E\land E)$ as an element $S=\sum b_i\hat\otimes a_i$ of $E^\ast E \hat\otimes_{E^\ast} E^\ast E$.
    We let $R=\sum a_i\hat\otimes b_i$. 
    One has $\mu = \mu T T = \mu\sum (b_i\otimes a_i)T = \mu T \sum(a_i\otimes b_i) = \mu RS$, so $R$ is just the inverse $R=S^{-1}$.
    \begin{lemma}
        This $R$ satisfies the defining relations of the \enquote{universal $R$-matrix} of a triangular structure on $E^\ast E$
        (see \cite{MR869575}*{section 10} or \cite{MR2654259}*{ch.~XI 2.1}).
        \begin{enumerate}
            \item $\opposite{\delta}(x) = R \delta(x) R^{-1}$
            \item $(\delta\otimes\id) R = R_{13}R_{23}$
            \item $(\id\otimes\delta) R = R_{13}R_{12}$
            \item $R_{12}R_{21} = \id$
        \end{enumerate}
    \end{lemma}
    Here one uses the $R_{i,j}\in E^\ast E \hat\otimes_{E^\ast} E^\ast E\hat\otimes_{E^\ast} E^\ast E$ given by
    \[ 
        R_{12} = \sum a_i\otimes b_i\otimes 1, \, 
        R_{13} = \sum a_i\otimes 1\otimes b_i, \, 
        R_{23} = \sum 1\otimes a_i\otimes b_i, \, 
        R_{21} = \sum b_i\otimes a_i\otimes 1.
    \] 
    \begin{proof}
        We will explain how to derive $\opposite{\delta}(x) = R \delta(x) R^{-1}$
        and leave the other claims as an exercise to the reader.

        We need to show $\delta(x) S = S\opposite{\delta}(x)$ for any $x\in E^\ast E$.
        The Künneth isomorphism makes this equivalent to $\mu\delta(x) S = \mu S \opposite{\delta}(x)$.
        The verification is then straightforward:
        \begin{align*}
            \mu\delta(x) S &= x\mu S = x\mu T, \\
            \mu S \opposite{\delta}(x) &= \mu T \opposite{\delta}(x) = \mu \delta(x) T = x \mu T.
        \end{align*} 
    \end{proof}

    As noted in \cite{MR869575}*{section 10, pt.~5)} the conditions (2) and (3) can be interpreted as statements 
    about the map 
    \begin{equation}\label{eq:drinphi}
        E_\ast E \rightarrow E^\ast E, \quad z \mapsto \sum \langle z,a_i\rangle \cdot b_i.
    \end{equation}
    Here $\langle z,a_i\rangle = \epsilon({a_i}_\ast(z)) = \mu(\id\land a_i)z \in E_\ast$ denotes the Kronecker pairing
    between $E_\ast E$ and $E^\ast E$.
    The conditions (2) and (3) then express that this map is suitably multiplicative and (anti-)comultiplicative.
    \begin{lemma}
        The map (\ref{eq:drinphi}) coincides with the map $\Phi:E_\ast E \rightarrow E^\ast E$.
    \end{lemma}
    \begin{proof} 
        Let $z\in E_\ast E$  
        be formally given as $z = \sum_j r_j\land s_j$. Then 
        \begin{align*} 
            \sum \langle z,a_i\rangle \cdot b_i
            &= \sum \mu\left(\left(\mu(\id\land a_i)z\right)\land b_i\right) 
            \\ &= \sum \mu_3\left(\id\land a_i\land b_i\right)\left(r_j\land s_j\land\id\right) 
            \\ &= \sum \mu_3\left(\id\land T\right) \left(r_j\land s_j\land\id\right)
            \\ &= \sum \mu_3\left(r_j\land \id\land s_j\right)\left(\id\land T\right)
        \end{align*}
        so for $x\in E^\ast X$ for some $X$ one has 
        \[ \sum \langle z,a_i\rangle \cdot b_i(x) = \sum r_jxs_j = \Phi_z(x). \] 
    \end{proof}
\end{subsection}

\begin{subsection}{The Drinfeld double of $\BFK$}
    Let $\BFK^\circ$ be the dual Hopf algebra $\BFK^\ast$, but
    with the flipped comultiplication and inverted antipode.
    Recall from \cite{MR2654259}*{III 2.4}
    that the Drinfeld double $D(\BFK)$ is a Hopf algebra that is characterized by
    \begin{enumerate}
        \item One has $D(\BFK) = \BFK \otimes \BFK^\circ$ as a coalgebra.
        \item The inclusions $\BFK\to D(\BFK)$ and $\BFK^\circ\to D(\BFK)$ are Hopf algebra homomorphisms.
        \item The universal $R$-matrix $R\in D(\BFK)\otimes D(\BFK)$ is the image
        of $\sum_i x_i \otimes x^i$ for a pair of dual bases
        $(x_i)\subset \BFK$ and $(x^j)\subset \BFK^\circ$.
    \end{enumerate}

    Our goal in this section is to establish a representation $D(\BFK)\rightarrow M\xi^\ast M\xi$.
    This amounts to the construction of two Hopf algebra homomorphisms 
    $\BFK\rightarrow M\xi^\ast M\xi$ and $\BFK^\circ \rightarrow M\xi^\ast M\xi$ with a suitable 
    commutation relation between the two.

    We already have $\BFK\rightarrow M\xi^\ast M\xi$ in place  as the composite of $\Phi$ with the inclusion 
    $\BFK \subset M\xi_\ast\opposite{M\xi}$. We have shown that the inclusion is a Hopf algebra map 
    and that $\Phi$ is multiplicative.  It remains to show that $\Phi$ is comultiplicative, too.
    \begin{lemma}
        The map $\Phi \colon M\xi_\ast\opposite{M\xi} \rightarrow M\xi^\ast M\xi$ is comultiplicative.
    \end{lemma}
    \begin{proof}
    Let $Z\in M\xi_\ast \opposite{M\xi}$ and assume
    formally
    \begin{align*}
        Z &= \sum e_k\land f_k, & \Delta(Z) &= \sum Z' \otimes Z'', \\
        Z' &= \sum a_i\land b_i, & Z'' &= \sum c_j\land d_j.
    \end{align*}
    Then
    \[\sum_{i,j} a_i\land b_ic_j \land d_j = \sum_k e_k\land 1\land f_k\]
    so
    \[\Phi_Z(xy) = \sum_k e_kxyf_k = \sum_{i,j} a_ixb_i c_jyd_j = \sum \Phi_{Z'}(x)\Phi_{Z''}(y).\]
    \end{proof}

    We next define $\Psi : \BFK^\circ \to M\xi^\ast M\xi$ using the isomorphism
    \[\begin{tikzcd}M\xi^\ast M\xi \ar[r, "\cong"]
        & \Hom_{M\xi_\ast}\left(M\xi_\ast \opposite{M\xi}, M\xi_\ast\right)\end{tikzcd}\]
    We have $M\xi_\ast \otimes \BFK \cong M\xi_\ast \opposite{M\xi}$ via $m\otimes a \mapsto m\cdot \lambda(a)$
    where $\lambda:\BFK\rightarrow M\xi_\ast\opposite{M\xi}$ denotes the inclusion.
    So given $\alpha\in \BFK^\circ$ we can define $\Psi(\alpha): M\xi\to M\xi$ via
    $$\epsilon\left(\Psi(\alpha)_\ast \lambda(b)\right) = \langle \alpha, b\rangle\quad\text{for all $b\in\BFK$.}$$
    \begin{lemma}\label{lem:bopincl}
        The map $\Psi: \BFK^\circ \to M\xi^\ast M\xi$
        is a map of Hopf algebras.
    \end{lemma}
    \begin{proof}
        We first show how to compute the induced map
        $\Psi(\alpha)_\ast: M\xi_\ast\opposite{M\xi} \to M\xi_\ast\opposite{M\xi}$
        from $\alpha$.
        Let $\rho:M\xi\to M\xi$ be any operation.
        We have a commutative diagram
        \[\begin{tikzcd}
            {M\xi_\ast\opposite{M\xi}} && {M\xi_\ast\opposite{M\xi} \otimes_{M\xi_\ast} M\xi_\ast\opposite{M\xi}} \\
            \\
            && {M\xi_\ast\opposite{M\xi}}
            \arrow["\Delta", from=1-1, to=1-3]
            \arrow["{\id \otimes \epsilon(\rho_\ast(-))}", from=1-3, to=3-3]
            \arrow["{\rho_\ast}"', from=1-1, to=3-3]
        \end{tikzcd}\]
        To see that this commutes assume $Z\in M\xi_\ast\opposite{M\xi}$ with $Z=\sum e_k\land f_k$
        and $\Delta(Z) = \sum_i (a_i\land b_i) \otimes \sum_j (c_j\land d_j)$.
        Then $\sum a_i\land b_ic_j\land d_j = \sum e_k\land 1\land f_k$ and one gets
        \begin{align*}
            \left(\id\otimes \epsilon(\rho_\ast(-))\right)\Delta(Z)
            &=
            \sum_{i,j} (a_i\land b_i) \otimes (c_j \rho d_j)
            \\ &\cong
            \sum_{i,j} a_i\land b_ic_j \rho d_j
            \\ &= (\id\land\mu)(\sum_{i,j} a_i\land b_ic_j \land \rho d_j)
            \\ &= (\id\land\mu)(\id\land\id\land\rho)(\sum_{i,j} a_i\land b_ic_j \land d_j)
            \\ &= (\id\land\mu)(\id\land\id\land\rho)(\sum_k e_k\land 1\land f_k)
            \\ &= \sum_k e_k\land \rho f_k
            \\ &= \rho_\ast(Z)
        \end{align*}
        where $\mu\colon{}M\xi\land M\xi\rightarrow M\xi$ is the multiplication and the $\cong$ line
        uses the isomorphism
        $M\xi_\ast\opposite{M\xi}\otimes_{M\xi_\ast} M\xi_\ast \cong M\xi_\ast\opposite{M\xi}$.

        We can now show that $\Psi$ is multiplicative:
        let $\alpha,\beta\in\BFK^\circ$ and $b\in \BFK$.
        The product $\alpha\beta$ is characterized by
        \[\langle \alpha\beta,b\rangle = \sum \langle\alpha,b'\rangle \langle\beta,b''\rangle\]
        The computation above, together with the comultiplicativity of $\lambda$ gives
        \[\Psi(\alpha)_\ast \lambda(b) = \sum \lambda(b') \langle \alpha,b''\rangle.\]
        We then have
        \begin{align*}
            \Psi(\beta)_\ast \Psi(\alpha)_\ast(\lambda(b))
            &= \sum \Psi(\beta)_\ast \lambda(b') \langle \alpha,b''\rangle
            \\ &= \sum \lambda(b') \langle \beta,b''\rangle \langle \alpha,b'''\rangle
            \\ &= \sum \lambda(b') \langle \beta\alpha, b''\rangle
            \\ &= \Psi(\beta\alpha)_\ast (\lambda(b))
        \end{align*}

        It remains to show that $\Psi$ is comultiplicative.
        Let $b_1,b_2\in\BFK$. We find
        \begin{align*}
            \Psi(\alpha)_\ast \lambda(b_1b_2)
            &= \sum \lambda(b_1') \lambda(b_2') \langle \alpha, b_1''b_2''\rangle
            \\&= \sum \lambda(b_1') \lambda(b_2') \langle \alpha', b_1''\rangle \langle \alpha'', b_2''\rangle
            \\&= \sum \lambda(b_1') \langle \alpha', b_1''\rangle \cdot \sum \lambda(b_2') \langle \alpha'', b_2''\rangle
            \\&= \sum \Psi(\alpha')_\ast(\lambda(b_1)) \cdot   \Psi(\alpha'')_\ast(\lambda(b_2))
        \end{align*}
        Here $\Delta \alpha = \sum \alpha' \otimes \alpha''$
        uses the coproduct in $\BFK^\ast$, not $\BFK^\circ$.

        Now note that for a general $\rho\colon{}M\xi\rightarrow M\xi$ we have
        \[\rho_\ast(xy) = \sum (\rho'')_\ast(x) \cdot (\rho')_\ast(y).\]
        This is due to the reversion of the multiplication in the second factor of $M\xi_\ast\opposite{M\xi}$.
        It follows that
        \begin{align*}
            \Psi(\alpha)_\ast \lambda(b_1b_2)
            &= \Psi(\alpha)_\ast (\lambda(b_1)\lambda(b_2))
            \\&= \sum \left(\Psi(\alpha)''\right)_\ast(\lambda(b_1))
                \cdot   \left(\Psi(\alpha)'\right)_\ast(\lambda(b_2))
        \end{align*}
        Hence
        $\Psi(\alpha') = \Psi(\alpha)''$, $\Psi(\alpha'') = \Psi(\alpha)'$ and
        $\Psi:\BFK^\circ \to M\xi^\ast M\xi$ is indeed comultiplicative.
    \end{proof}

    \begin{thm}\label{lem:doublehomo}
        The map $\rho \colon{} D(\BFK) \rightarrow M\xi^\ast M\xi$
        with $b_1 \otimes b_2 \mapsto \Phi(b_1) \circ \Psi(b_2)$ is a map of Hopf algebras.
    \end{thm}
    \begin{proof}
        The map is clearly comultiplicative, so it only remains to verify that the
        commutation relation between some $b\in \BFK \subset D(\BFK)$
        and $c\in \BFK^\circ \subset D(\BFK)$ is preserved under $\rho$.

        It is straightforward to determine the commutation relations between a 
        $\psi = \Psi(\alpha)$ with $\alpha\in \BFK^\circ$ 
        and a $\Phi_z$ for $z\in M\xi_\ast\opposite{M\xi}$:
        Let $\Delta^{(3)} \alpha = \sum \alpha'\otimes \alpha''\otimes \alpha'''$ in $B^\ast$ (not $B^\circ$),
        $\psi'=\Psi(\alpha')$, $\psi''=\Psi(\alpha'')$, $\psi'''=\Psi(\alpha''')$
        so that $\psi(uvw) = \sum \psi'''(u) \psi''(v) \psi'(w)$.
        Let formally $z= \sum p_i\land q_i$. Then
        \begin{align*} 
            \psi(\Phi_z(x)) &= \sum \psi(p_ixq_i) = \sum \psi'''(p_i) \psi''(x) \psi'(q_i)
            = \sum \Phi_{(\psi'''\land \psi')(z)}(\psi''(x)).
        \end{align*}
        It remains to convince oneself that the same relation exists in the Drinfeld double $D(\BFK)$. 
        Unfortunately, most published accounts of the Drinfeld doubling construction 
        are too timid to work out explicit formulas for the multiplication in $D(\BFK)$. 
        A laudable exception is \cite{MR3709143} where the product formula is given in section 3.3.
        Letting $\BFK^\circ\rightsquigarrow A$, $1\rightsquigarrow x$, $\psi\rightsquigarrow a$, 
        $z\rightsquigarrow y$ and $1\rightsquigarrow b$ in their formula (3.1) gives 
        \[ (1\otimes\psi) (z\otimes 1) = \sum \langle\psi',z'''\rangle \langle\psi''',\chi^{-1}(z')\rangle z'' \otimes \psi'' \]
        where $\chi$ denotes the antipode in $\BFK$.
        We entrust it to the skeptical reader to verify that the $\langle\psi',z'''\rangle \langle\psi''',\chi^{-1}(z')\rangle z''$
        can indeed be neatly aligned with our $\Phi_{(\psi'''\land \psi')(z)}$.
    \end{proof}
\end{subsection}

\end{section}


\bibliographystyle{unsrt}
\bibliography{cnbib}

\def\cprime{$'$}
\begin{bibdiv}
\begin{biblist}

\bib{MR402720}{book}{
      author={Adams, J.~F.},
       title={Stable homotopy and generalised homology},
      series={Chicago Lectures in Mathematics},
   publisher={University of Chicago Press, Chicago, Ill.-London},
        date={1974},
      review={\MR{402720}},
}

\bib{MR4368864}{book}{
      author={Bahturin, Yuri},
       title={Identical relations in {L}ie algebras},
      series={De Gruyter Expositions in Mathematics},
   publisher={De Gruyter, Berlin},
        date={2021},
      volume={68},
        ISBN={978-3-11-056557-7; 978-3-11-056665-9; 978-3-11-056570-6},
         url={https://doi.org/10.1515/9783110566659},
        note={Second edition [of 0886063]},
      review={\MR{4368864}},
}

\bib{MR2200854}{article}{
      author={Brouder, Christian},
      author={Frabetti, Alessandra},
      author={Krattenthaler, Christian},
       title={Non-commutative {H}opf algebra of formal diffeomorphisms},
        date={2006},
        ISSN={0001-8708},
     journal={Adv. Math.},
      volume={200},
      number={2},
       pages={479\ndash 524},
      review={\MR{2200854}},
}

\bib{MR2484353}{article}{
      author={Baker, Andrew},
      author={Richter, Birgit},
       title={Quasisymmetric functions from a topological point of view},
        date={2008},
        ISSN={0025-5521},
     journal={Math. Scand.},
      volume={103},
      number={2},
       pages={208\ndash 242},
      review={\MR{2484353}},
}

\bib{MR3243390}{incollection}{
      author={Baker, Andrew},
      author={Richter, Birgit},
       title={Some properties of the {T}hom spectrum over loop suspension of complex projective space},
        date={2014},
   booktitle={An alpine expedition through algebraic topology},
      series={Contemp. Math.},
      volume={617},
   publisher={Amer. Math. Soc., Providence, RI},
       pages={1\ndash 12},
      review={\MR{3243390}},
}

\bib{MR3438568}{article}{
      author={Bukhshtaber, V.~M.},
      author={Ustinov, A.~V.},
       title={Coefficient rings of formal groups},
        date={2015},
        ISSN={0368-8666},
     journal={Mat. Sb.},
      volume={206},
      number={11},
       pages={19\ndash 60},
         url={https://doi.org/10.4213/sm8523},
      review={\MR{3438568}},
}

\bib{MR869575}{article}{
      author={Drinfel\cprime~d, V.~G.},
       title={Quantum groups},
        date={1986},
        ISSN={0373-2703},
     journal={Zap. Nauchn. Sem. Leningrad. Otdel. Mat. Inst. Steklov. (LOMI)},
      volume={155},
      number={Differentsial\cprime naya Geometriya, Gruppy Li i Mekh. VIII},
       pages={18\ndash 49, 193},
         url={https://doi.org/10.1007/BF01247086},
      review={\MR{869575}},
}

\bib{MR2132761}{article}{
      author={Gelfand, Israel},
      author={Gelfand, Sergei},
      author={Retakh, Vladimir},
      author={Wilson, Robert~Lee},
       title={Quasideterminants},
        date={2005},
        ISSN={0001-8708},
     journal={Adv. Math.},
      volume={193},
      number={1},
       pages={56\ndash 141},
         url={https://doi.org/10.1016/j.aim.2004.03.018},
      review={\MR{2132761}},
}

\bib{MR1398918}{incollection}{
      author={Gelfand, Israel},
      author={Retakh, Vladimir},
       title={Noncommutative {V}ieta theorem and symmetric functions},
        date={1996},
   booktitle={The {G}elfand {M}athematical {S}eminars, 1993--1995},
      series={Gelfand Math. Sem.},
   publisher={Birkh\"{a}user Boston, Boston, MA},
       pages={93\ndash 100},
      review={\MR{1398918}},
}

\bib{MR2683228}{incollection}{
      author={Hazewinkel, Michiel},
       title={Niceness theorems},
        date={2009},
   booktitle={Advances in combinatorial mathematics},
   publisher={Springer, Berlin},
       pages={79\ndash 125},
         url={https://doi.org/10.1007/978-3-642-03562-3_5},
      review={\MR{2683228}},
}

\bib{MR2987372}{book}{
      author={Hazewinkel, Michiel},
       title={Formal groups and applications},
   publisher={AMS Chelsea Publishing, Providence, RI},
        date={2012},
        ISBN={978-0-8218-5349-8},
         url={https://doi.org/10.1090/chel/375},
        note={Corrected reprint of the 1978 original},
      review={\MR{2987372}},
}

\bib{MR77542}{article}{
      author={Lazard, Michel},
       title={Lois de groupes et analyseurs},
        date={1955},
        ISSN={0012-9593},
     journal={Ann. Sci. \'{E}cole Norm. Sup. (3)},
      volume={72},
       pages={299\ndash 400},
         url={http://www.numdam.org/item?id=ASENS_1955_3_72_4_299_0},
      review={\MR{77542}},
}

\bib{MR1243637}{book}{
      author={Montgomery, Susan},
       title={Hopf algebras and their actions on rings},
      series={CBMS Regional Conference Series in Mathematics},
   publisher={Published for the Conference Board of the Mathematical Sciences, Washington, DC; by the American Mathematical Society, Providence, RI},
        date={1993},
      volume={82},
        ISBN={0-8218-0738-2},
         url={https://doi.org/10.1090/cbms/082},
      review={\MR{1243637}},
}

\bib{morava2020renormalization}{misc}{
      author={Morava, Jack},
       title={Renormalization groupoids in algebraic topology},
        date={2020},
}

\bib{MR253350}{article}{
      author={Quillen, Daniel},
       title={On the formal group laws of unoriented and complex cobordism theory},
        date={1969},
        ISSN={0002-9904},
     journal={Bull. Amer. Math. Soc.},
      volume={75},
       pages={1293\ndash 1298},
         url={https://doi.org/10.1090/S0002-9904-1969-12401-8},
      review={\MR{253350}},
}

\bib{MR860042}{book}{
      author={Ravenel, Douglas~C.},
       title={Complex cobordism and stable homotopy groups of spheres},
      series={Pure and Applied Mathematics},
   publisher={Academic Press, Inc., Orlando, FL},
        date={1986},
      volume={121},
        ISBN={0-12-583430-6; 0-12-583431-4},
      review={\MR{860042}},
}

\bib{MR1656075}{article}{
      author={Schauenburg, Peter},
       title={On the braiding on a {H}opf algebra in a braided category},
        date={1998},
     journal={New York J. Math.},
      volume={4},
       pages={259\ndash 263},
         url={http://nyjm.albany.edu:8000/j/1998/4_259.html},
      review={\MR{1656075}},
}

\bib{MR3709143}{article}{
      author={Schrader, Gus},
      author={Shapiro, Alexander},
       title={Dual pairs of quantum moment maps and doubles of {H}opf algebras},
        date={2017},
        ISSN={0021-8693},
     journal={J. Algebra},
      volume={492},
       pages={74\ndash 89},
         url={https://doi.org/10.1016/j.jalgebra.2017.08.021},
      review={\MR{3709143}},
}

\bib{MR1800719}{incollection}{
      author={Takeuchi, Mitsuhiro},
       title={Survey of braided {H}opf algebras},
        date={2000},
   booktitle={New trends in {H}opf algebra theory ({L}a {F}alda, 1999)},
      series={Contemp. Math.},
      volume={267},
   publisher={Amer. Math. Soc., Providence, RI},
       pages={301\ndash 323},
         url={https://doi.org/10.1090/conm/267/04277},
      review={\MR{1800719}},
}

\bib{sagemath}{manual}{
      author={{The Sage Developers}},
       title={{S}agemath, the {S}age {M}athematics {S}oftware {S}ystem ({V}ersion 10.3)},
        date={2024},
        note={{\tt https://www.sagemath.org}},
}

\bib{MR2654259}{book}{
      author={Turaev, Vladimir~G.},
       title={Quantum invariants of knots and 3-manifolds},
     edition={revised},
      series={De Gruyter Studies in Mathematics},
   publisher={Walter de Gruyter \& Co., Berlin},
        date={2010},
      volume={18},
        ISBN={978-3-11-022183-1},
         url={https://doi.org/10.1515/9783110221848},
      review={\MR{2654259}},
}

\bib{MR1273649}{article}{
      author={Van~Daele, A.},
      author={Van~Keer, S.},
       title={The {Y}ang-{B}axter and pentagon equation},
        date={1994},
        ISSN={0010-437X},
     journal={Compositio Math.},
      volume={91},
      number={2},
       pages={201\ndash 221},
         url={http://www.numdam.org/item?id=CM_1994__91_2_201_0},
      review={\MR{1273649}},
}

\bib{MR655040}{book}{
      author={Wilson, W.~Stephen},
       title={Brown-{P}eterson homology: an introduction and sampler},
      series={CBMS Regional Conference Series in Mathematics},
   publisher={Conference Board of the Mathematical Sciences, Washington, D.C.},
        date={1982},
      volume={48},
        ISBN={0-8219-1699-3},
      review={\MR{655040}},
}

\end{biblist}
\end{bibdiv}

\end{document}